\documentclass{daj}

\usepackage{amsmath}
\usepackage{amsthm}
\usepackage{amssymb}
\usepackage{mathtools}
\usepackage{tikz}
\usetikzlibrary{calc}
\tikzset{vtx/.style={circle, fill, inner sep=1.2pt}}
\tikzset{bigvtx/.style={circle, fill, inner sep=2.1pt}}
\usepackage{enumitem}

\newtheorem{theorem}{Theorem}
\newtheorem{lemma}[theorem]{Lemma}
\newtheorem{proposition}[theorem]{Proposition}
\theoremstyle{definition}
\newtheorem{definition}[theorem]{Definition}
\newtheorem{conjecture}[theorem]{Conjecture}
\theoremstyle{remark}
\newtheorem{remark}[theorem]{Remark}

\numberwithin{equation}{section}
\numberwithin{theorem}{section}

\newcommand{\N}{\mathbb{N}}
\newcommand{\R}{\mathbb{R}}
\newcommand{\Z}{\mathbb{Z}}
\newcommand{\cala}{\mathcal{A}}
\newcommand{\calf}{\mathcal{F}}
\newcommand{\calg}{\mathcal{G}}
\newcommand{\calp}{\mathcal{P}}
\newcommand{\calr}{\mathcal{R}}
\newcommand{\diam}{\mathrm{diam}}

\DeclarePairedDelimiter{\ceil}{\lceil}{\rceil}
\DeclarePairedDelimiter{\floor}{\lfloor}{\rfloor}

\dajAUTHORdetails{%
  title = {Possible Sizes of Sumsets}, %% please capitalize all significant words
  author = {Isaac Rajagopal},
  plaintextauthor = {Isaac Rajagopal},
}   %%% END \dajAUTHORdetails

\dajEDITORdetails{%
   year={2026},
   number={2},
   received={11 November 2025},   % received date: example: 7 January 2017
   published={16 July 2026},  % published date
   doi={10.19086/da.165102},       % XXX = number of paper, e.g. da006 for paper#6
}   %%% END \dajEDITORdetails  

\begin{document}

\begin{frontmatter}[classification=text]
%% EDITOR: this will force the keywords to appear right after the Abstract.
%%   If the abstract is too long and would force the keywords off the
%%   front page, please comment out % [classification=text] above
%%   This way the keywords will be floated on the bottom of the first page
%%   even though the Abstract spills over to the next page.

%%% AUTHOR: Title goes here.  This line is optional.  You must use it
%%   if title has footnote attached or requires nontrivial typesetting,
%%   e.g., inclusion of linebreaks to force nice layout.
\title{Possible Sizes of Sumsets} %% please capitalize all significant words

%%% AUTHOR:
%%% List all authors. If you wish, place grant acknowledgements in \thanks.
%%% In brackets include a short tag for each author.
\author[isaac]{Isaac Rajagopal}

%%% AUTHOR: Abstract goes here
\begin{abstract}
   Nathanson introduced the \emph{range of cardinalities of $h$-fold sumsets} \[\mathcal{R}(h,k) \coloneq  \{|hA|:A \subset \Z \text{ and }|A| = k\}.\] Following a remark of \text{Erd\H{o}s} and Szemer\'edi that determined the form of $\mathcal{R}(h,k)$ when $h=2$, Nathanson asked what the form of $\mathcal{R}(h,k)$ is for arbitrary $h, k \in \mathbb{N}$. For $h \in \mathbb{N}$, we prove there is some constant $k_h \in \mathbb{N}$ such that if $k > k_h$, then $\mathcal{R}(h,k)$ is the entire interval $\left[hk-h+1,\binom{h+k-1}{h}\right]$ except for a specified set of $\binom{h-1}{2}$ numbers. Moreover, we show that one can take $k_3 = 2$.
\end{abstract}
\end{frontmatter}

\section{Introduction}
Let $A \subset \Z$ be a set of integers. For $h \in \N$, write \[hA \coloneq  \{a_1+\cdots+a_h:a_i \in A\}\] to denote the $h$-fold sumset of $A$. Note that the $a_i$ need not be distinct. Nathanson defined the set $\calr(h,k)$ of the possible sizes of these $h$-fold sumsets when $|A| = k$. 

\begin{definition}[\cite{nathanson25Problems}]
    Let $h, k \in \N$. Define the \emph{range of cardinalities of $h$-fold sumsets } \[\mathcal{R}(h,k) \coloneq  \{|hA| : A \subset \Z \text{ and } |A| = k\}.\]
\end{definition}

Nathanson \cite{nathanson25Problems} posed the problem of determining the exact values in $\mathcal{R}(h,k)$ for arbitrary $h,k \in \N$. In this paper, we will determine the form of $\mathcal{R}(h,k)$ for fixed $h \in \N$ and large $k$, and for $h=3$ and all $k \in \N$, specifically answering Nathanson's Problems 8 and 9 in \cite{nathanson25Problems}.

\begin{figure}[h]
\begin{center}
\begin{tikzpicture}[xscale=0.7, yscale=0.7, every node/.style={font=\small}]
\draw[very thick,fill=lightgray!70] (-0.3,-0.7) -- (-0.3,-4.75) -- (3.75,-0.7) -- (-0.3,-0.7); 
  \node at (5,0) {37};
  \foreach \j in {0,...,5} {
    \pgfmathtruncatemacro{\val}{38 + \j}
    \node at (\j,-1) {\val};
  }
  \foreach \j in {0,...,5} {
    \pgfmathtruncatemacro{\val}{44 + \j}
    \node at (\j,-2) {\val};
  }
  \foreach \j in {0,...,5} {
    \pgfmathtruncatemacro{\val}{50 + \j}
    \node at (\j,-3) {\val};
  }
  \foreach \j in {0,...,5} {
    \pgfmathtruncatemacro{\val}{56 + \j}
    \node at (\j,-4) {\val};
  }
  \foreach \j in {0,...,5} {
    \pgfmathtruncatemacro{\val}{62 + \j}
    \node at (\j,-5) {\val};
  }
  \node at (2.5,-5.5) {$\vdots$};
  \node at (1.0,-5.5) {$\vdots$};
  \node at (4.0,-5.5) {$\vdots$};
  \foreach \j in {0,...,4} {
    \pgfmathtruncatemacro{\val}{920 + \j}
    \node at (\j,-6.5) {\val};
  } 
\end{tikzpicture}
\end{center}
\caption{We trivially have $\mathcal{R}(6,7) \subseteq [37,924]$. The gray triangle represents the numbers in $\Delta_{6,7}$ and Theorem~\ref{mainthm} shows that there are no elements of $\mathcal{R}(6,7)$ in $\Delta_{6,7}$. Conjecture~\ref{mainconj} says that $\calr(6,7)$ is all of the listed numbers outside the gray triangle.}\label{figure1}
\end{figure}
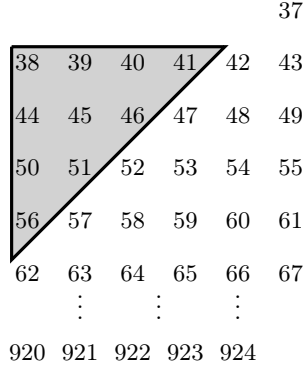

The first work related to $\calr(h,k)$ was by Erd\H{o}s and Szemer\'edi \cite{ES83}, who remarked that ``it is easy to see that'' $\mathcal{R}(2,k) = \left[2k-1,\binom{k+1}{2}\right]$ for all $k \in \N$. Nathanson \cite{nathanson25Problems} determined the exact form of $\calr(h,3)$  for any 
$h \in \N$. Nathanson's recent program has involved finding specific values in $\calr(h,k)$ and studying the frequency with which certain sumset sizes show 
up, with particular attention to $k=4$ 
\cite{nathanson2025additivesumsetsizestetrahedral,nathanson2025explicitsumsetsizesadditive,nathanson25Problems,nathanson2025triangulartetrahedralnumberdifferences,nathanson2025compressioncomplexitysumsetsizes,obryant2025nathansonstriangularnumberphenomenon}. With fixed $h>2$ and $k$ increasing, the problem of determining $\calr(h,k)$ remained fully open.

For real numbers $a$ and $b$, let $[a,b]$ denote the interval of integers from $\ceil{a}$ to $\floor{b}$, including endpoints. It is not hard to show \cite[Theorem 1.3]{Nathanson96} that \begin{equation}\label{eqtrivialbound}
    \mathcal{R}(h,k) \subseteq \left[hk-h+1,\binom{h+k-1}{h}\right],
\end{equation} and that the endpoints of this interval are achieved by the arithmetic progression $A = [0,k-1]$ and the geometric series $A = \{1,h+1,(h+1)^2,\ldots,(h+1)^{k-1}\}$.
Let $m = hk-h+1$. Tang--Xing \cite{TX19} and Schinina \cite{Schinina25} independently found that there are no elements of $\mathcal{R}(h,k)$ in $[m+1,m+h-2]$. Tang--Xing also found that there are no elements of $\mathcal{R}(h,k)$ in $[m+h+1,m+2h-3] \cup [m+2h+1,m+3h-4]$ if $k \geq 5$. We extend this trio of intervals into a set of $\min\{h,k\}-2$ intervals, which we call $\Delta_{h,k}$ because they  geometrically form a triangle when $h \leq k$, as shown in Figure~\ref{figure1}.

\begin{definition}\label{defdelta}
    Let $m = hk-h+1$. Then define 
    \[\Delta_{h,k} \coloneq  \bigcup_{\ell =0}^{\min\{h,k\}-3}[m+\ell h+1,m+\ell h+(h-2-\ell)].\]
\end{definition}
\begin{theorem}\label{mainthm} For all $h,k \in \N$, we have
    \[\mathcal{R}(h,k) \cap \Delta_{h,k} = \emptyset.\]
\end{theorem}

Figure~\ref{figure1} shows an illustrated example of the points ruled out by Theorem~\ref{mainthm}. Tang--Xing's \cite{TX19} theorem describes the intervals in $\Delta_{h,k}$ with $\ell = 0$ and $\ell = 1$ and $\ell = 2$, which are the uppermost three rows of the triangle in Figure \ref{figure1}, while Schinina's \cite{Schinina25} theorem describes the interval in $\Delta_{h,k}$ with $\ell = 0$, which is the uppermost row of the triangle in Figure \ref{figure1}. Our main theorem states that for general $h$ and large $k$, all other numbers in the bounding interval outside $\Delta_{h,k}$ are in $\calr(h,k)$.

\begin{theorem}\label{thmgeneralh}
    Fix $h \in \N$. There exists a constant $k_h \in \N$ such that if $k > k_h$, then \[\calr(h,k) = \left[hk-h+1,\binom{h+k-1}{h}\right] \setminus \Delta_{h,k}.\]
\end{theorem}

\begin{remark}\label{remarkasterisk}
Tracking constants in our proof shows that one can take $k_h \leq 2^{20h^2}$. We have not attempted to optimize this quantity further because it appears to be quite far from the truth.
\end{remark}
Numerical experiments for small $h$ and $k$ suggest that perhaps the conclusion of Theorem \ref{thmgeneralh} holds for all $k> h$. 

\begin{conjecture}\label{mainconj}
Let $h \in \N$, and let $k>h$. Then \[\calr(h,k) = \left[hk-h+1,\binom{h+k-1}{h}\right] \setminus \Delta_{h,k}.\]
\end{conjecture}

When $k < h-1$, the conclusion of Conjecture~\ref{mainconj} never holds; this follows from results of Lev \cite{Lev96}. See Section \ref{secthgeqk} for a discussion of what happens when $k \leq h$.

For any fixed $h$, Theorem~\ref{thmgeneralh} reduces Conjecture~\ref{mainconj} to a finite problem, albeit a computationally intensive one. We are able to prove that $k_3 = 2$ suffices, which implies Conjecture~\ref{mainconj} when $h = 3$; note that $\Delta_{3,k} = \{3k-1\}$.

\begin{theorem}\label{thmh3}
    If $k>2$, then \[\calr(3,k) = \{3k-2\} \cup \left[3k,\binom{k+2}{3}\right].\]
\end{theorem}

In contrast with previous work, our proofs of Theorem~\ref{thmgeneralh} and Theorem~\ref{thmh3} are nonconstructive; for a single desired sumset size, our proof does not give an explicit construction attaining it. Instead, we use continuity-type arguments to prove existence. We construct a sequence of sets $A_i$ such that, even though the quantities $|hA_{i}|$ vary non-monotonically, we can guarantee that the entire set $\{|hA_i|\}$ is an interval. To study $|hA|$, we also study the distribution of the tuples $(|A|, |2A|, |3A|,\ldots,|hA|)$ for various $k$-element sets $A$. Even though we do not characterize exactly which tuples are possible, we extract enough structural information to characterize the possible values of $|hA|$. This opens up fruitful directions for new research, to control these tuples more precisely. Another key part of our construction is combining dense sets, meaning sets with small doubling, and sparse sets, which look like Sidon sets or $B_h$ sets, to get sets in the middle. 

\begin{remark}
    We can also define $N(h,k)$ to be the minimal $N \geq 0$ such that \[\calr(h,k) = \{|hA|:A \subset [0,N] \text{ and }|A| = k\}.\] By general considerations, Nathanson \cite[Theorem 8]{nathanson2025compressioncomplexitysumsetsizes} proved that $N(h,k) \leq 4(8h)^{k-1}$. Recently, \mbox{ChatGPT 5.5 Pro}, with light prompting from Timothy Gowers and the author, optimized the constructions in this paper to prove that $N(h,k)$ is polynomial in $k$. In a recent blog post \cite{GR26}, Gowers and the author describe the ChatGPT construction and intuition for why it works, detail how to show the explicit bound $N(h,k) \leq k^{10h^3}$ for $k$ sufficiently large, and provide the fully ChatGPT-written proof that $N(h,k)$ is polynomial in $k$.
\end{remark}

\subsection{Roadmap}
In Section~\ref{sectexcluding}, we prove Theorem~\ref{mainthm} by using results from \cite{Lev96} which generalize Freiman's $3k-3$ theorem to control the sizes of sumsets from the diameter of sets. We then prove Theorem~\ref{thmh3} before Theorem~\ref{thmgeneralh} because the proof of Theorem~\ref{thmh3} contains simpler versions of some of the ideas in the proof of Theorem~\ref{thmgeneralh}. In Section~\ref{secth3}, we prove Theorem~\ref{thmh3} using a nonconstructive inductive argument involving understanding the possible tuples $(|2A|,|3A|)$ for all sets $A \subset \Z$ with $|A| = k$. The bulk of the paper is in Section~\ref{sectgeneralh}, where we prove Theorem~\ref{thmgeneralh}. This is done using two separate constructions: In Section~\ref{sectlemmaquad}, we construct sets $A$ with smaller values of $|hA|$ using a partially inductive construction; in Section~\ref{sectlemmamain}, we construct sets $A$ with larger values of $|hA|$ by combining dense sets that resemble arithmetic progressions with sparse sets that resemble geometric sequences. Section~\ref{sectfuturedirections} discusses possible future directions of research and open problems.

\section{Ruling out Numbers from \texorpdfstring{$\calr(h,k)$}{R(h,k)}}\label{sectexcluding}
In this section, we will prove Theorem \ref{mainthm} quickly using results from \cite{Lev96}. We begin by defining the diameter of a set $A$. 
\begin{definition}
    For $A \subset \Z$ finite, define its \emph{diameter} $\diam(A)$ as 1 less than the minimum length of an arithmetic progression containing $A$.
\end{definition}
Since affine linear transformations of $A$ preserve both $|hA|$ and $\diam(A)$, we may always assume that $\min(A) = 0$ and $\max(A) = \diam(A)$ and $\gcd(A) = 1$. Write $A = \{0,a_1,\ldots,a_{k-1}\}$ with $a_i < a_{i+1}$ and $\gcd(a_1,\ldots,a_{k-1})=1$. Let $d \coloneq a_{k-1} = \diam(A)$. Now, we state a result from \cite{Lev96} about how $\diam(A)$ controls $|hA|$.

\begin{theorem}[\cite{Lev96}]\label{thmlev}
    For any $A \subset \Z$ with $|A|=k\geq 2$ and $d = \diam(A)$ and $h \geq 2$, \begin{equation}\label{eqlev1}
        |hA| \geq |(h-1)A| + \min\{d,h(k-2)+1\}.
    \end{equation} Let $q \in \N$ satisfy $q(k-2)+ 1 \leq d \leq (q+1)(k-2)+1$. By applying (\ref{eqlev1}) inductively, \begin{equation}\label{eqlev2}
        |hA| \geq |h\{0,1,2,\ldots,k-2,d\}| = \begin{cases}
       \frac{h(h+1)}{2}(k-2)+h+1 &\text{ if } h \leq q \\
        \frac{q(q+1)}{2}(k-2)+q+1 + (h-q)d &\text{ if } h \geq q.
    \end{cases}
    \end{equation}
\end{theorem}

When $h=2$ and $q=1$, (\ref{eqlev2}) becomes Freiman's $3k-3$ theorem (see \cite[Theorem 1.13]{Nathanson96}). Using Theorem \ref{thmlev}, we can prove Theorem \ref{mainthm}.

\begin{proof}[Proof of Theorem~\ref{mainthm}]
    If $h \in \{1,2\}$ or $k \in \{1,2\}$ then $\Delta_{h,k} = \emptyset$, so let $h,k \geq 3$. Let $A$ be a set of $k$ integers, with $d = \diam(A)$. It suffices to show that $|hA| \not \in \Delta_{h,k}$.

    \textbf{Case 1:} $d < 2k-3$.

    Let $\ell \in [0,\min\{h,k\}-3]$. Suppose $|hA| \geq m+\ell h+1 = h(k-1 +\ell)+2$. For cardinality reasons, $hA \not \subseteq [0,h(k-1+\ell)]$, so $d \geq k-\ell$. In Theorem \ref{thmlev}, we have $q = 1$ so therefore (\ref{eqlev2}) becomes \[|hA| \geq k+(h-1)d \geq k + (h-1)(k-\ell) = m+\ell h+(h-\ell-1).\] Therefore, $|hA|$ is not in the interval $[m+\ell h+1,m+ \ell h+(h-  \ell-2)]$. So $|hA| \not \in \Delta_{h,k}$.

    \textbf{Case 2:} $d \geq 2k-3$.

    Theorem \ref{thmlev}, specifically (\ref{eqlev1}), implies that $|iA| -|(i-1)A| \geq  2k-3$ for all $i \geq 2$. Applying this $h-1$ times and using that $|A| = k$, then \begin{equation}\label{eqboundhA}
        |hA| \geq k + (h-1)(2k-3) = 2hk-3h-k+3. 
    \end{equation} As $\ell \leq k-3$ for all elements of $\Delta_{h,k}$, \[\max(\Delta_{h,k}) \leq m+(k-3) h + (h-2-(k-3)) = 2hk-3h-k+2.\]
    Therefore, $|hA| > \max(\Delta_{h,k})$ so $|hA|\not \in \Delta_{h,k}$. 
\end{proof}

\section{Finding \texorpdfstring{$\calr(3,k)$}{R(3,k)}}\label{secth3}

We begin by defining a useful graph, which we will use to formulate our inductive hypothesis.

\begin{definition}
    Define the \emph{graph representation} of a set $A \subset \Z$ as $g(A)\coloneq  (|2A|,|3A|)$.
    Define the \emph{$k$-sumset graph} $\mathcal{G}_k$ to have vertex set \[ V_k = \{g(A): A \subset \Z \text{ and } |A| = k\}.\] Place an edge between distinct vertices $(x,y)$ and $(x',y')$ if $|x-x'| \leq 1$ and $|y-y'| \leq 1$.
\end{definition}

\begin{figure}
    \centering
\begin{tikzpicture}[scale=0.35]

\draw[step=1, black!20, thin] (25,60) grid ($(10,15)-(0.5,0.5)$);
\draw[step=5, black!90] (25,60) grid ($(10,15)-(0.5,0.5)$);

\draw (25,60) rectangle (10,15);

\foreach \x in {15,20,...,60} {
    \node[scale=1] at (8.5,\x) {\x};
}
\foreach \x in {10,15,...,25} {
    \node[scale=1] at (\x,13.5) {\x};
}

\node[scale=1.2] at (17.5,11) {$|2A|$};

\node[scale=1.2] at (6,37.5) {$|3A|$};

\foreach \x in {
{(11,16)}, {(15,22)},  {(16,24)}, {(16,25)}, {(17,25)}, {(15,26)}, {(17,26)}, {(15,27)}, {(17,27)}, {(15,28)}, {(17,28)}, {(18,28)},  {(17,29)}, {(18,29)}, {(19,29)}, {(17,30)}, {(18,30)}, {(19,30)}, {(17,31)}, {(18,31)}, {(19,31)}, {(17,32)}, {(18,32)}, {(19,32)}, {(17,33)}, {(18,33)},
{(19,33)}, {(17,34)}, {(18,34)}, {(19,34)}, {(18,35)}, {(19,35)}, {(18,36)}, {(19,36)}, {(20,36)}, {(18,37)},
{(19,37)}, {(20,37)}, {(18,38)}, {(19,38)}, {(20,38)}, {(18,39)}, {(19,39)}, {(20,39)}, {(18,40)}, {(19,40)},
{(20,40)}, {(19,41)}, {(20,41)}, {(19,42)}, {(20,42)}, {(21,42)}, {(19,43)}, {(20,43)}, {(21,43)}, {(19,44)},
{(20,44)}, {(21,44)}, {(20,45)}, {(21,45)}, {(20,46)}, {(21,46)}, {(20,47)}, {(21,47)}, {(20,48)}, {(21,48)},
{(20,49)}, {(21,49)}, {(20,50)}, {(21,50)}, 
{(21,51)}, {(21,52)}, {(21,53)}, {(21,54)}, {(21,55)}%
} {
    \node[blue, bigvtx] at \x {};
}

\foreach \x in {
    {(12, 18)}, {(13, 19)}, {(13, 20)}, {(14, 21)}, {(14, 22)}, {(15, 23)}, {(15, 24)},
    {(15, 25)}, {(16, 26)}, {(16, 27)}, {(16, 28)}, {(16, 29)}, {(16,30)}, {(16, 31)}, {(17,32)}, {(17,33)}, {(17,34)}, {(18,35)}, {(18,36)}, {(18,37)}, {(18,38)}, {(19,39)}, {(19,40)}, {(19,41)}, {(19, 42)}, {(19,43)}, {(19, 44)}, {(20,45)}, {(20,46)}, {(20,47)}, {(20,48)}, {(20,49)}, {(20,50)}, {(21,51)}, {(21,52)}, {(21,53)}, {(21,54)}, {(21,55)}, {(21,56)}}
{
    \node[red, bigvtx] (red\x) at \x {};
}

\draw[black, thick, line width = 1.5pt] (12, 18)--(13, 19)--(13, 20)--(14, 21)--(14, 22)--(15, 23)--(15, 24)--(15, 25)--(16, 26)--(16, 27) -- (16, 28)--(16, 29)--(16, 30)--(16, 31)--(17, 32)--(17, 33)--(17, 34)--(18, 35)--(18, 36)--(18, 37)--(18, 38)--(19, 39)--(19, 40)--(19, 41)--(19, 42)--(19, 43)--(19, 44)--(20, 45)--(20, 46)--(20, 47)--(20, 48)--(20, 49)--(20, 50)--(21, 51)--(21, 52)--(21, 53)--(21, 54)--(21, 55)--(21, 56);
\end{tikzpicture}
\caption{The blue and red points are all of the vertices of $\mathcal{G}_{6}$.  These are the possible values of ${g(A) = (|2A|,|3A|)}$ when $|A| = 6$.  The red points indicate the path constructed in the proof of Theorem~\ref{thmh3} to connect $(12,18)$ to $(21,56)$.
}\label{fig2}
\end{figure}
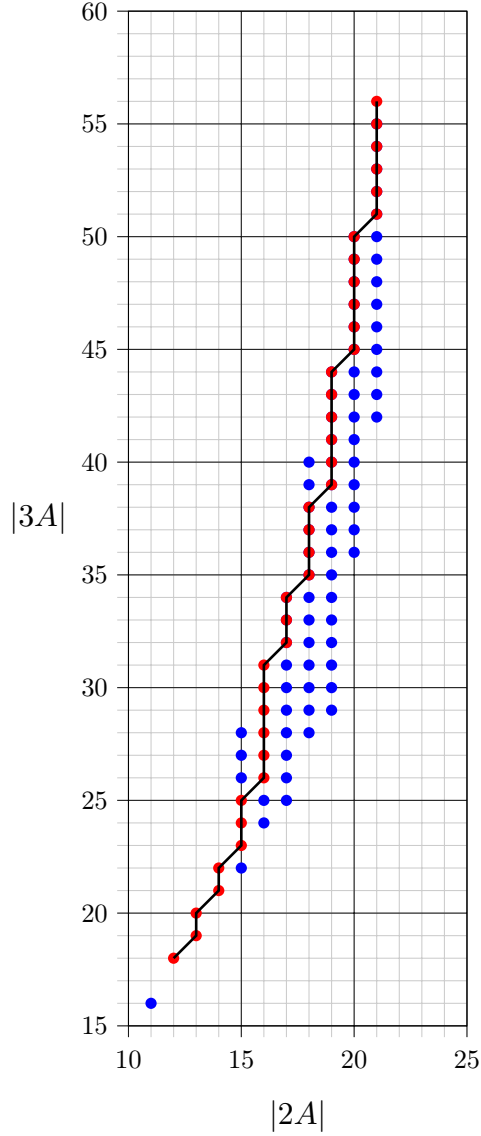

As an example, the vertex set $V_6$ is shown in Figure~\ref{fig2}. To prove Theorem~\ref{thmh3}, we will show inductively that for $k \geq 3$, the vertices $(2k,3k)$ and $\big(\binom{k+1}{2},\binom{k+2}{3}\big)$ are connected in $\calg_k$; an example path is shown in Figure~\ref{fig2}. Before starting the induction, we must prove that a particular part of the path between these vertices exists.

\begin{lemma}\label{lemmanew}
    For $k \geq 4$, the points $(2k,3k)$ and $(3k-2,6k-5)$ are connected by a path in $\calg_k$.
\end{lemma}
\begin{proof}
    We will explicitly construct a path $\mathcal{P}$ from $(2k,3k)$ to $(3k-2,6k-5)$. Define \[A_{a} = \{0\} \cup[2,k-1]\cup\{k+a\}, \; B_b = [0,1] \cup [3,k-1] \cup \{k+b\}.\] Then explicit computations yield that \[g(A_{a}) = (|2A_{a}|,|3A_{a}|) = \begin{cases}
        (2k+a,3k+2a) &\text{ if } 0\leq a \leq k-3 \\
        (3k-3,5k-5) &\text{ if } a = k-2 \\
        (3k-2,4k-2 + a) &\text{ if } k-1 \leq a \leq 2k-4\\
        (3k-2,6k-5) &\text{ if } a = 2k-2,
    \end{cases}\]
    and 
    \[g(B_{b}) = (|2B_{b}|,|3B_b|) = \begin{cases}
        (2k+1+b,3k+1+2b) &\text{ if } 0 \leq b \leq k-4 \\
        (3k-2,5k-4) &\text{ if } b = k-2, k >4.
    \end{cases}\]

    Also, define $B_2' = \{0,1,4,5\}$, and compute $g(B_2') = (9,16)$. Let $B_{k-2}^*$ denote $B_{k-2}$ if $k > 4$ and $B_2'$ if $k = 4$. Using the notation $g(C_1,\ldots,C_n) = g(C_1),\ldots,g(C_n)$, then we can check that \[\mathcal{P} = g(A_0,B_0,A_1,B_1,\ldots,A_{k-4},B_{k-4},A_{k-3},A_{k-2},B_{k-2}^*,A_{k-1},A_k,A_{k+1},\ldots,A_{2k-5},A_{2k-4},A_{2k-2})\] gives a path from $(2k,3k)$ to $(3k-2,6k-5)$ in $\calg_k$, as desired.
\end{proof}

We are now ready to prove Theorem~\ref{thmh3} by inducting on $k$, with inductive hypothesis the existence of a path between $(2k,3k)$ and $\big(\binom{k+1}{2},\binom{k+2}{3}\big)$ in $\calg_k$.

\begin{proof}[Proof of Theorem~\ref{thmh3}]
    We induct on $k$. Our inductive hypothesis is that the vertices $(2k,3k)$ and $\big(\binom{k+1}{2},\binom{k+2}{3}\big)$ are connected by a path in $\calg_k$. To prove Theorem~\ref{thmh3} from this, it suffices to consider a path $\mathcal{P}$ from $(2k,3k)$ to $\big(\binom{k+1}{2},\binom{k+2}{3}\big)$. Looking at the $y$-values of the vertices in $\mathcal{P}$ implies that $\left[3k,\binom{k+2}{3}\right] \subseteq \mathcal{R}(3,k)$. Combining this with Theorem~\ref{mainthm}, which excludes $3k-1$ from $\calr(3,k)$, proves Theorem~\ref{thmh3}.

    The base case is $k=3$, where it is easy to check that $g(\{0,1,3\})$ and $g(\{0,1,4\})$ form a path from $(6,9)$ to $(6,10)$ in $\calg_3$. Now, let $k > 3$. By the inductive hypothesis, assume that $(2(k-1),3(k-1))$ and $\big(\binom{k}{2},\binom{k+1}{3}\big)$ are connected in $\calg_{k-1}$. We will show that $(2k,3k)$ and $\big(\binom{k+1}{2},\binom{k+2}{3}\big)$ are connected in $\calg_k$.

    Let $A \subseteq [0,d]$ be a set with $|A| = k-1$ and $0,d \in A$. Define $A_1 = A \cup \{3d\}$ and $A_2 = A \cup \{3d+1\}$ to be sets with $|A_1| = |A_2| = k$. Then \begin{equation}\label{eqanalogy}3A_2 = 3A\cup (3d+1+2A) \cup (6d+2+A) \cup \{9d+3\}.\end{equation} Since $hA \subseteq [0,hd]$ for $h \in \{1,2,3\}$, these four sets are disjoint and \[|3A_2| = |3A|+|2A|+(k-1)+1 = |3A|+|2A|+k.\] With $A_2$ replaced by $A_1$, we have \[3A_1 = 3A\cup (3d+2A) \cup (6d+A) \cup \{9d\}.\] These four sets overlap in exactly one place, since $3d \in 3A\cap (3d+2A)$. So \[|3A_1| = |3A|+|2A|+k-1.\] By similar logic, we can check that \[|2A_1|=|2A_2| = |2A|+k.\] 
    
    For $i \in \{1,2\}$, we can define maps $\phi_i:V_{k-1} \to V_k$ by $\phi_i(g(A)) = g(A_i)$. Then \[\phi_i((x,y)) = \left(x+k,x+y+k+\begin{cases}
        -1 & \text{ if } i = 1 \\
        0 &\text{ if } i = 2 
    \end{cases}\right).\]
    For any set $S \subseteq V_{k-1}$, let \[\phi(S) = \phi_1(S) \cup \phi_2(S).\] 
    
    Let $\phi(S) \subseteq \calg_k$ refer to the induced subgraph of $\phi(S)$. Suppose $(x,y)$ and $(x',y')$ are adjacent in $\calg_{k-1}$. By checking each value of $x-x'$ and $y-y'$ in $\{-1,0,1\}$, we can deduce that the graph $\phi(\{(x,y),(x',y')\}) \subseteq \calg_k$ (consisting of up to four vertices) is connected. 
    
    Now, suppose $(x,y)$ and $(x',y')$ are connected in $\calg_{k-1}$, so there is a path $\mathcal{P}$ from $(x,y)=(x_1,y_1)$ to $(x',y')=(x_n,y_n)$ with vertices $(x_1,y_1),(x_2,y_2),\ldots,(x_n,y_n)$. So $\phi(\{(x_i,y_i),(x_{i+1},y_{i+1})\}) \subseteq \calg_k$ is connected for all $1 \leq i < n$. Therefore, \[\phi(\{(x_1,y_1),(x_2,y_2),\ldots,(x_n,y_n)\})\subseteq \calg_k\] is connected, which means that $\phi_2((x,y))$ is connected to $\phi_2((x',y'))$ within $\calg_k$.

    Let $(x,y) = (2(k-1),3(k-1))$ and $(x',y') = \big(\binom{k}{2},\binom{k+1}{3}\big)$, which are connected in $\calg_{k-1}$ by the inductive hypothesis. The previous paragraph demonstrates that $\phi_2((x,y)) = (3k-2,6k-5)$ and $\phi_2((x',y'))= \big(\binom{k+1}{2},\binom{k+2}{3}\big)$ are connected in $\calg_k$. Lemma~\ref{lemmanew} gives that $(2k,3k)$ and $(3k-2,6k-5)$ are also connected in $\calg_k$. By concatenating the paths from $(2k,3k)$ to $(3k-2,6k-5)$ and from $(3k-2,6k-5)$ to $\big(\binom{k+1}{2},\binom{k+2}{3}\big)$, we are done.
\end{proof}

\section{Finding \texorpdfstring{$\calr(h,k)$}{R(h,k)} with General \texorpdfstring{$h$}{h}}\label{sectgeneralh}

\subsection{Overview}
We introduce the statements of two main propositions, Proposition~\ref{lemmaQuadratic} and Proposition~\ref{lemmaMain}, which we will prove in Sections \ref{sectlemmaquad} and \ref{sectlemmamain}, respectively. Proposition~\ref{lemmaQuadratic} shows we can achieve the the smaller values in $\calr(h,k)$, while Proposition~\ref{lemmaMain} shows the same for the larger values. 

\begin{proposition}\label{lemmaQuadratic}
Fix $h \in \N$, and let $k > 4^hh^2$. Then \[\left[hk-h+1,\frac{k^h}{4^{2h^2}}\right]\setminus \Delta_{h,k} \subseteq \calr(h,k).\]
\end{proposition}

\begin{proposition}\label{lemmaMain}
    Fix $h \geq 4$, and $\varepsilon \in \R_{>0}$. There exists some constant $k_{h,\varepsilon} \in \N$ such that for $k > k_{h,\varepsilon}$, we have \[\left[\varepsilon k^h,\binom{h+k-1}{h}\right] \subseteq \mathcal{R}(h,k).\]
\end{proposition}

Theorem~\ref{thmgeneralh} follows immediately from these propositions.

\begin{proof}[Proof of Theorem~\ref{thmgeneralh}]
    With $h=1$, the theorem is obvious. With $h=2$, the theorem follows from \cite{ES83}. With $h=3$, the theorem follows from Theorem \ref{thmh3}. Let $h \geq 4$. Let $\varepsilon = \frac{1}{4^{2h^2}}$. Setting $k_h = \max\{k_{h,\varepsilon},4^hh^2\}$ and letting $k>k_h$, we have \[\left[hk-h+1,\binom{h+k-1}{h}\right]\setminus \Delta_{h,k} \subseteq \calr(h,k)\] by Propositions \ref{lemmaQuadratic} and \ref{lemmaMain}. By Theorem~\ref{mainthm}, $\Delta_{h,k} \cap \calr(h,k) = \emptyset$, so we are done.
\end{proof}

\subsection{Achieving Small Sumset Sizes}\label{sectlemmaquad}
In this section, we will prove Proposition \ref{lemmaQuadratic}. We begin by defining $(h,d)$-filling sets. Then, we will construct sets $A$ as the union of an $(h,d)$-filling set of size $k-1$ for some $d$ and a single number. We will prove Proposition \ref{lemmaQuadratic} by showing that varying $d$ and the single number will cause $hA$ to achieve all the desired sizes. 

\begin{definition}
    For $h \in \N$, a set $A \subseteq [0,d]$ is \emph{$(h,d)$-filling} if $hA = [0,hd]$.
\end{definition}

If $A$ is $(h,d)$-filling, then $d \in A$, so it follows that $iA = [0,id]$ for all $i > h$. Therefore, $A$ being $(h,d)$-filling implies that $A$ is $(i,d)$-filling for all $i \geq h$. Also, if $A$ is $(h,d)$-filling and $a \in [0,d]$, then $A \cup \{a\}$ is $(h,d)$-filling. We first show the existence of $(h,d)$ filling sets of size $k$ for all $d \ll \varepsilon k^h$, for some small fixed $\varepsilon>0$.

\begin{lemma}\label{lemmaauxilliary}
    Let $h \geq 2$ and $k \geq 4^{h+1}$ and \[d \in \left[k-1,\frac{k^h}{4^{(h^2+h)/2}}\right].\] Then there exists an $(h,d)$-filling set $A$ with $|A| = k$ such that $[0,\frac{k}{8}] \subset A$.
\end{lemma}
\begin{proof}
    Notice that $\frac{k^h}{4^{(h^2+h)/2}} \geq k$ since $k \geq 4^{h+1}$. We induct on $h$. The base case is $h = 2$.

    \begin{figure}
            \centering
            \begin{tikzpicture}[scale=0.11]
                \foreach \j in {0,...,10} {
                    \node at (\j,0)[circle,fill,inner sep=1pt]{};
                    \node at (90+\j,0)[circle,fill,inner sep=1pt]{};
                }
                \foreach \j in {2,...,8} {
                    \node at (10*\j,0)[circle,fill,inner sep=1pt]{};
                }
                \node at (0,-2){0};
                \node at (10,-2){10};
                \node at (50,-2){50};
                \node at (90,-2){90};
                \node at (100,-2){100};
            \end{tikzpicture}
            \caption{An example of $A'$ from the proof of Lemma~\ref{lemmaauxilliary} with $d = 100$.}
            \label{figA'}
            
            \end{figure}
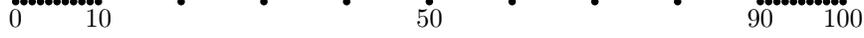

    With $h=2$, then $d \in [k-1,\frac{k^2}{64}]$. Let $\tilde{d} = \lfloor \sqrt{d}\rfloor$. Consider the set \[A' = [0, \tilde{d}] \cup \{i\tilde{d}: i \in \Z\text{ and } 0 \leq i \leq d/\tilde{d}\} \cup [d-\tilde{d},d].\] See Figure~\ref{figA'} for a picture of this construction. It is easy to check that $A'$ is $(2,d)$-filling. Because $\tilde{d} \leq \frac{k}8$, we get $|A'| < \frac{k}{2}$. Now, create $A$ from $A'$ by adding the smallest $k - |A'|$ numbers in $[0,d]\setminus A'$ to $A'$. This is possible since $d \geq k-1$. Then $A$ remains $(2,d)$-filling. Since $k - |A'|>\frac{k}{2}$, we have that $[0,\frac{k}{8}]\subseteq [0,\frac{k}{2}] \subseteq A$.
        
    Let $h>2$ and assume that the lemma holds for $h-1$. We will prove it for $h$. Let $k \geq 4^{h+1}$. If \[d \leq \frac{k^{h-1}}{4^{(h^2-h)/2}},\] then we can let $A$ be the $(h-1,d)$-filling set constructed by induction, and we are done. So assume that \[d \in \left[\frac{k^{h-1}}{4^{(h^2-h)/2}},\frac{k^{h}}{4^{(h^2+h)/2}}\right].\]
    Then \[\floor*{\frac{4d}{k}} \in \left[\frac{k^{h-2}}{4^{(h^2-h)/2-1}}-1,\frac{(k/4)^{h-1}}{4^{(h^2-h)/2}}\right].\] 
    Let $\tilde{d} =\left\lfloor\frac{4d}{k}\right\rfloor$ and let $\tilde{k} = \left\lceil \frac{k}{4} \right\rceil$, and notice $\tilde{d} \leq \tilde{k}^{h-1}/4^{(h^2-h)/2}$. If $\tilde{d} \geq \tilde{k}-1$, the inductive hypothesis supplies an $(h-1,\tilde{d})$-filling set $\tilde{A}$ with $|\tilde{A}| = \tilde{k}$. If $\tilde{d} \leq \tilde{k}-1$, let $\tilde{A} = [0,\tilde{d}]$ which is clearly $(h-1,\tilde{d})$-filling. Now, let 
    \[A' = \tilde{A} \cup \{i\tilde{d}: i \in \Z \text{ and }0 \leq i \leq d/\tilde{d}\} \cup (d-\tilde{d}+\tilde{A}).\] 
    Each of the components of this union has size at most $\frac{k}{3.5}$, so $|A'| \leq \frac{7k}{8}$. 

    Let us now show that $A'$ is $(h,d)$-filling. Since $\tilde{A}$ is $(h-1,\tilde{d})$-filling, we have \[(h-1)\left(\tilde{A} \cup  (d-\tilde{d}+\tilde{A})\right) = \bigcup_{j=0}^{h-1} \Big(j(d-\tilde{d}) + [0,(h-1)\tilde{d}]\Big).\] It is then not hard to check that $hA' = [0,hd]$. To construct $A$, add the smallest $k-|A'|$ numbers between $0$ and $d$ to $A'$, which is possible since $d \geq k-1$. Then $A$ remains $(h,d)$-filling. Because $k-|A'| \geq \frac{k}{8}$, we have  $[0,\frac{k}{8}] \subseteq A$.
\end{proof}

We are now ready to prove Proposition~\ref{lemmaQuadratic} by taking $A$ as the union of an $(h,d)$-filling set and a number.

\begin{proof}[Proof of Proposition~\ref{lemmaQuadratic}]
With $h=1$ the proposition is vacuously true, so let $h \geq 2$. We will define two parameters $d$ and $i$ to vary, and we will show that varying these parameters gives all the desired sets. 

Let $k-2 \leq d \leq \frac{(k-1)^h}{4^{(h^2+h)/2}}$. Lemma~\ref{lemmaauxilliary} supplies a set $B$ with $|B| = k-1$ such that $B$ is $(h,d)$-filling and $[0,\frac{k-1}{8}] \subseteq B$. 
Since $k>4^hh^2$, we have $[0,h^2] \subseteq B$. Let $i \in [1,h]$, and define \[A = \{0\} \cup (i+B).\] Then, since $hi + hB = [hi,hi+hd]$, we have \[hA = \{0\} \cup (i +B) \cup (2i+2B) \cup (3i + 3B)  \cup \cdots \cup [hi,hi+hd].\] If $1 \leq m \leq h$, then $(mi + mB) \subseteq [mi,mi + md]$. Since $[0,h^2]\subseteq B$, we have $[i,hi] \subseteq (i + B)$, so \[hA = \{0\} \cup [i,hi+hd].\]
Thus, $|hA| = h(i+d)-i+2$. Define $\ell$ by $i+d = \ell+k$. So the conditions $d \in \left[k-2,\frac{(k-1)^h}{4^{(h^2+h)/2}}\right]$ and $i \in [1,h]$ become \[\ell \in \left[-1,\frac{(k-1)^h}{4^{(h^2+h)/2}}+h-k\right] \text{ and } i \in \left[\max\left\{1,\ell+k-\frac{(k-1)^h}{4^{(h^2+h)/2}}\right\}, \min\{h,\ell +2\}\right].\] 
Substituting $i+d$ for $\ell + k$ gives that $|hA| = hk + h\ell -i+2$. Then, recalling that $m = hk-h+1$, for any $\ell \in \left[-1,\frac{(k-1)^h}{4^{(h^2+h)/2}}-k\right]$, we have \[\calr(h,k) \supseteq hk+h\ell-[1,\min\{h,\ell +2\}]+2 = 
m + h\ell + [\max\{1,h-\ell-1\},h].\] By plugging in the definition of $\Delta_{h,k}$, this gives that \[\left[hk-h+1,h\left(\frac{(k-1)^h}{4^{(h^2+h)/2}}-k\right)\right]\setminus \Delta_{h,k} \subseteq \calr(h,k),\] and this immediately implies Proposition~\ref{lemmaQuadratic}.
\end{proof}

\subsection{Achieving Large Sumset Sizes}\label{sectlemmamain}

\subsubsection{Overview}
In this section, we will prove Proposition~\ref{lemmaMain}. We will first define a way to combine two sets $A_1$ and $A_2$ to get a set with $|A_1| + |A_2|$ elements, which we call the disjoint union $A_1 \sqcup A_2$. Our construction of sets $A$ to fill the desired interval will be given as $A = B \sqcup C$, where $|B| = b$ and $|C| = c = k-b$, with $b \in [3,k-\floor{k^{0.7}}]$. Here, $B$ is a dense set ($|hB|$ is small for $h>1$) and $C$ is a sparse set ($|hC|$ is large for $h>1$). 

We define $B_{j,b} = [0,b-2] \cup \{h(b-2)+2-j\}$ with $j \in [1,(h-1)(b-2)+1]$ to be the union of an arithmetic progression and a single larger element. We also define the sets $S_{m,s_m} = \{0,1,m, m^2,\ldots,m^{s_m-2}\}$ and $T_{m,t_m}~=~\{1,m,m^2,\ldots,m^{t_m-1}\}$ with powers of $m$, and then construct \[C = \bigsqcup_{m=3}^h S_{m,s_m} \sqcup \bigsqcup_{m=2}^h T_{m,t_m}.\] For $m \in [3,h]$, we vary $s_m \in [2,\floor{c^{0.9}}]$ and $t_{m-1}\in [2,\floor{c^{0.8}}]$, with $t_h = c - \sum_{m=3}^h(s_m+t_{m-1})$. Let \begin{equation}\label{eqdefofA}
    A = B_{j,b} \sqcup C = B_{j,b}\sqcup\bigsqcup_{m=3}^h S_{m,s_m} \sqcup \bigsqcup_{m=2}^h T_{m,t_m}.
\end{equation} We will show that varying the set of $2h-2$ parameters $b, j, s_m, t_{m-1}$ for $m \in [3,h]$ forces $|hA|$ to achieve all the desired sizes. 

We begin by stating a key auxiliary lemma.

\begin{lemma}\label{lemmahypercubeIVT}

    Let $d\in \N$, and let $n_i \in \N$ for all $i \in [1,d]$. Let $f:[1,n_1]\times \cdots \times [1,n_d] \to \Z$ be a function which is nonincreasing in each entry. For $i \in [1,d]$, let \[\delta_i = \max (f(x_1,\ldots,x_{i-1},x_i,x_{i+1},\ldots,x_d) - f(x_1,\ldots,x_{i-1},x_i+1,x_{i+1},\ldots,x_d)) \geq 0\] and \[\Delta_i = \min (f(x_1,\ldots,x_{i-1},1,x_{i+1},\ldots,x_d) - f(x_1,\ldots,x_{i-1},n_i,x_{i+1},\ldots,x_d)) \geq 0.\] Suppose that $\delta_1 = 1$ and $\delta_{i}\leq \Delta_{i-1}$ for all $i \in [2,d]$. Then $\mathrm{im}(f)$ is an interval.
\end{lemma}
To apply Lemma \ref{lemmahypercubeIVT}, let $d = 2h-3$, with \begin{equation}\label{eqxi}(x_1,x_2,x_3,\ldots,x_{d}) = (s_h-1,t_{h-1}-1,s_{h-1}-1,t_{h-2}-1,\ldots,s_{3}-1,t_2-1,j).\end{equation} Fix $b \in [3,k-\floor{k^{0.7}}]$ and $c = k-b$. Set \begin{equation}\label{eqvalni}
n_i = \begin{cases}
    \floor{c^{0.9}}-1 &\text{ if } i \text{ is odd and } i <2h-3 \\
    \floor{c^{0.8}}-1 &\text{ if } i \text{ is even} \\
    (h-1)(b-2)+1 &\text{ if } i = 2h-3.
\end{cases}\end{equation} Let $f(x_1,\ldots,x_{d}) = hA$, where $A$ is defined as in (\ref{eqdefofA}). We will then show that $hA$ is nonincreasing in each $x_i$, that $\delta_1 = 1$, and that $\Delta_{i-1} \geq \delta_i$ for $i \in [2,d]$. Let $\cala_b$ to be the set of all $A$ with $b$ fixed and $j, s_m, t_{m-1}$ for $m \in [3,h]$ varying in the prescribed intervals. By applying Lemma \ref{lemmahypercubeIVT}, the set $h\cala_b \coloneq  \{|hA|: A \in \cala_b\}$ is an interval. 

We will then show that $h\cala_b \cap h\cala_{b-1} \neq \emptyset$, meaning that the intervals $h\cala_b$ overlap for different values of $b$. Thus, with $\cala = \displaystyle \bigcup_{b=3}^{k-\floor{k^{0.7}}} \cala_b$, the set $h\cala\coloneq \{|hA|: A \in \cala\}$ is an interval. Finally, it suffices to show that there is a value smaller than $\varepsilon k^h$ in $h\cala_{k-\floor{k^{0.7}}}$ and that $\binom{h+k-1}{h} \in h\cala_3$. Since $h\cala$ is an interval, it contains the whole interval we want, which finishes the proof.

In our construction, we can think of each $x_i$ as a parameter to vary, where increasing each $x_i$ adds more relations of sums of elements of $A$, decreasing $|hA|$. For smaller values of $i$, varying $x_i$ gives finer control over the value of $|hA|$, while varying $x_i$ for larger values of $i$ gives coarser control over the value of $|hA|$.

We first prove Lemma \ref{lemmahypercubeIVT}. In Section \ref{sssection: disjointunions}, we define disjoint unions and a useful generating function, and show that the generating function behaves multiplicatively with respect to disjoint unions. We also define notions of $ O[[\cdot]]$ and $\Theta_+[[\cdot]]$ for the size of generating functions. In Section~\ref{sssection: components}, we approximate the sumset sizes of the components $B_{j,\ell}$, $S_{m,\ell}$, and $T_{m,\ell}$. In Section \ref{sssection: finish}, we show that our construction satisfies the conditions of Lemma \ref{lemmahypercubeIVT}, and then use this to prove Proposition \ref{lemmaMain}. Throughout this section, fix $h \geq 4$.

\begin{proof}[Proof of Lemma \ref{lemmahypercubeIVT}]
    We show that $f([1,n_1]\times \cdots \times [1,n_i] \times \{y_{i+1}\}\times \cdots \times \{y_d\})$ is an interval for any $i \in [1,d]$ and $y_{i+1} \in [1,n_{i+1}],\ldots,y_d\in[1,n_d]$ by induction on $i$. With $i = 1$, this is true since $\delta_1=1$ and \[f(a,y_2,\ldots,y_d) - \delta_1 \leq f(a+1,y_2,\ldots,y_d) \leq f(a,y_2,\ldots,y_d)\] for any $a\in [1,n_{1}-1].$ 
    
    Now, assume the inductive hypothesis is true for $i-1$, and we will show it to be true for $i$. Define $I_{x_{i}} \coloneq  f([1,n_1]\times \cdots \times [1,n_{i-1}] \times \{x_{i}\}\times \{y_{i+1}\}\times \cdots \times \{y_d\})$ for $x_{i} \in [1,n_{i}]$, which is an interval by our inductive hypothesis. Then \begin{equation}\label{eqIa}
        f([1,n_1]\times \cdots \times [1,n_{i-1}] \times [1,n_{i}] \times \{y_{i+1}\} \times \cdots \times \{y_d\}) = \bigcup_{a=1}^{n_{i}} I_{a}.
    \end{equation} Since $f$ is nonincreasing in each entry, $\min I_{a} \geq \min I_{a+1}$ for all $a \in [1,n_{i}-1]$. Also, by the definitions of $\delta_{i}$ and $\Delta_{i-1}$, we have \[ f(1,1,1,\ldots,n_{i-1},a,y_{i+1},\ldots,y_{d}) \leq f(1,1,1,\ldots,1,a+1,y_{i+1},\ldots,y_{d})+\delta_{i}-\Delta_{i-1}.\] The value of $f$ on the left side is in $I_a$ and the value of $f$ on the right side is in $I_{a+1}$. Since $\delta_{i}-\Delta_{i-1} \leq 0$, we have $\min I_{a} \leq \max I_{a+1}$. Hence, $I_a$ and $I_{a+1}$ overlap. So both sides of (\ref{eqIa}) are intervals, and we are done.
\end{proof}
\subsubsection{Disjoint Unions and Generating Functions}\label{sssection: disjointunions}
First, we note that working in $\Z$ and $\Z^r$ are equivalent.
\begin{lemma}\label{lemmadimension}
    For $\eta, r, \kappa \in \N$, we have \[\{(|A|,|2A|,\ldots,|\eta A|):A \subset 
    \Z^r \textnormal{ and } |A| = \kappa \} = \{(|A|,|2A|,\ldots,|\eta A|):A \subset \Z \textnormal{ and }  |A| = 
\kappa \}.\]
\end{lemma}
\begin{proof}
We induct on $r$. With $r = 1$ this is trivial. Given $A \subset \Z^{r}$, let $M = \max\{\|a\|_\infty:a \in A\}$. Define a function $\phi: \Z^r \to \Z^{r-1}$ to be the linear map which takes $(x_1,\ldots,x_{r-1},0) \to (x_1,\ldots,x_{r-1})$ and takes ${(0,0,\ldots,0,1) \mapsto (\lceil 3\eta M\rceil,0,0,\ldots,0)}$. Let $A' = \phi(A)$. It is easy to check that $\phi$ is a Freiman isomorphism of order $\eta$, so we have $|iA'| = |iA|$ for $1 \leq i \leq \eta$.
\end{proof}
We now define a disjoint union of sets and a generating function.

\begin{definition}\label{defDisjointUnion}
    Let $A \subset \Z^r$ and $B \subset \Z^s$ with $|A| = a$ and $|B| = b$. Define the \emph{disjoint union} $A \sqcup B$ as the set of $a+b$ points in $\Z^{r+s+2}$ given by \[A \sqcup B = (A,0,1,0) \cup (0,B,0,1) \subset \Z^r \times \Z^s \times \Z \times \Z.\]
\end{definition}

\begin{definition}\label{defGenFunc}
    Define $0A = \{0\}$. For $A \subset \Z^r$ finite, define the generating function \[\mathcal{F}_A(z) \coloneq \sum_{\eta=0}^{\infty} |\eta A|z^\eta.\]
\end{definition}
We now prove that the generating function behaves multiplicatively with respect to disjoint unions.

\begin{lemma}\label{lemmagenfunc} For $A \subset \Z^r$ and $B\subset \Z^s$ finite,
    \[\mathcal{F}_{A \sqcup B}(z) = \mathcal{F}_A(z) \cdot \mathcal{F}_B(z).\]
\end{lemma}
\begin{proof}
    Elements of $\eta (A \sqcup B)$ are formed by adding up $i$ elements from $(A,0,1,0)$ and $\eta -i$ elements from $(0,B,0,1)$, which yields $|iA||(\eta -i)B|$ possible sums. Sums of $\eta$ elements with different values of $i$ cannot be equal, since they will differ in their last two coordinates. Varying the value of $i$ between $0$ and $\eta $ gives \[|\eta (A \sqcup B)| = |\eta A|+|(\eta -1)A||B| + |(\eta -2)A||2B|+\cdots+|\eta B|,\] so we are done by applying Definition \ref{defGenFunc}.
\end{proof}

Notice that $\calf_{A \sqcup(B \sqcup C)}(z) = \calf_{(A \sqcup B)\sqcup C}(z)$. Whenever taking a disjoint union of more than two sets, we do so in an arbitrary order, as in such circumstances we will only care about the data in the generating function, which we know to be equal. In particular, with $A = \displaystyle\bigsqcup_{i=1}^m {A_{i}}$, we have \[\calf_A(z) = \prod_{i=1}^m \calf_{A_{i}}(z).\]

Recall that $h$ is fixed; we now consider asymptotics as $k$ grows to infinity. Given a function $f: \N \to \Z$ and $i \in \R_{\geq 0}$, we let the expressions $f(a) \in O(a^i)$ and $f(a) \in \Theta(a^i)$\footnote{If $i <0$ then $f(a) \in O(a^i)$ or $\Theta(a^i)$ if $f(a) = 0$ for all $a \in \N$.} have their usual meanings. Say that $f(a) \in \Theta_+(a^i)$ if $f(a) \in \Theta(a^i)$ and there is some $a_0 \in \N$ such that $f(a) > 0$ if $a > a_0$. 
 Now, we define the notions of $O[[\cdot]]$ and $\Theta_+[[\cdot]]$ and $\equiv$ that we will use to approximate the sizes of generating functions.

\begin{definition}
    We say that a power series $\calf(z) = 1 + a_1z+a_2z^{2} + a_3z^3 + \cdots$ is in $O[[az]]$ (respectively $\Theta_+[[az]]$) if $a_i \in O(a^i)$ (respectively $\Theta_+(a^i)$) for $i \leq h$. For $\calg(z)= 1 + b_1z+b_2z^{2} + b_3z^3 + \cdots$ another power series, we say $\calf(z) \equiv \calg(z)$ if $a_i = b_i$ for all $i \leq h$.
\end{definition}

Note that for $\calf(z) \in O[[az]]$ or $\Theta_+[[az]]$, we require $\calf(z)$ to have constant term $a_0 = 1$. We now state a lemma about how the sizes of power series vary under multiplication.

\begin{lemma}\label{lemmaProductsofGenFuncs}
    Fix $r \in \N$. If $\calf(z) \in \Theta_+[[az]]$ and $b = o(a)$ and $\calg^{(j)}(z) \in O[[bz]]$ for all $j \in [1,r]$, then \[\calp(z) \coloneq \calf(z)\cdot \prod_{j=1}^r\calg^{(j)}(z) \in \Theta_+[[az]].\]
\end{lemma}
\begin{proof}
    We induct on $r$. First let $r = 1$. Let $\calf(z) = 1 +a_1z+a_2z^{2} + a_3z^3 + \cdots$ for $a_i \in \Theta_+(a^i)$. Let $\calg^{(1)}(z) = 1 +b_1z+b_2z^{2} + b_3z^3 + \cdots$ for $b_i \in O(b^i)$. Let $\calp(z) = 1+p_1z+p_2z^2+p_3z^3 + \cdots $. For $\eta \leq h$, we have \[p_\eta = \sum_{\nu=0}^\eta a_\nu b_{\eta-\nu}.\] The fastest growing term will be $a_\eta \in \Theta_+(a^\eta)$. All other terms will be in $O(a^\eta) \cdot \frac{b}{a}$, and will thus be dwarfed by $a_\eta$. Thus, $p_{\eta} \in \Theta_+(a^\eta)$. 

    Now, assume the statement is true for $r -1$. Let \[\calp'(z) \coloneq \calf(z)\cdot \prod_{j=1}^{r-1}\calg^{(j)}(z) \in \Theta_+[[az]]\] by the inductive hypothesis. By repeating the argument of $r = 1$ on $\calp(z) = \calp'(z)\calg^{(r)}(z)$, we get that $\calp(z) \in \Theta_+[[az]]$.
\end{proof}

\subsubsection{Sumset Sizes of \texorpdfstring{$B_{j,\ell}$}{Bjl}, \texorpdfstring{$S_{m,\ell}$}{Sml}, and \texorpdfstring{$T_{m,\ell}$}{Tml}}\label{sssection: components}
We begin with our construction of dense sets $B_{j,\ell}$, which are given by the union of an arithmetic progression and an integer, and state their properties. 
\begin{definition}
    Let $\ell \geq 3$, and $s = (h-1)(\ell-2)+1$. Define the set $B_{j,\ell} \coloneq  [0,\ell-2] \cup \{h(\ell-2)+2-j\}$ for $j \in [1,s]$.
\end{definition}
\begin{lemma}\label{lemmadense}
    For $\ell \geq 3$, the sets $B_{1,\ell},\ldots,B_{s,\ell}$ satisfy the following: \begin{enumerate}[label = (\alph*)]
        \item $0 \leq |\eta B_{j-1,\ell}|-|\eta B_{j,\ell}| \leq h$ for all $j \in [2,s]$ and $\eta \in [2,h]$;
        \item $|\eta B_{j,\ell}| < h^2\ell$ for all $j \in [1,s]$ and $\eta \in [2,h]$; 
        \item $\calf_{B_{1,3}}(z) \equiv (1-z)^{-3}.$
        \item If $\ell >3$, then $\calf_{B_{1,\ell}}(z) \equiv (1-z)^{-1}\calf_{B_{s,\ell-1}}(z).$
    \end{enumerate}
\end{lemma}
\begin{proof}
    An exact formula for $|\eta B_{j,\ell}|$, due to \cite{Lev96}, is given in (\ref{eqlev2}). This immediately implies (a). For (b), notice that $B_{j,\ell} \subseteq [0,h(\ell-1)]$ and $\eta \leq h$, so $\eta B_{j,\ell} \subseteq [0,h^2(\ell-1)]$. Therefore $|\eta B_{j,\ell}| < h^2\ell$. For (c), the binomial theorem gives $(1-z)^{-3} = \sum_{\eta = 0}^\infty \binom{\eta+2}{2}z^\eta$. By definition, $B_{1,3} = \{0,1,h+1\}$. For $1 \leq \eta \leq h$, the identity $|\eta B_{1,3}| = \binom{\eta+2}{2}$ follows from setting $q$ to be $h$ and $h$ to be $\eta$ and $k$ to be 3 in (\ref{eqlev2}).  
    
    For part (d), let $d = h(\ell-2)+1$. Then $B_{s,\ell-1} = [0,\ell-2]$ and $B_{1,\ell} = [0,\ell-2] \cup \{d\}$. Let $\eta \in [1,h]$. In a similar way to (\ref{eqanalogy}), \begin{equation}\label{eqtempor}
        \eta B_{1,\ell} = \bigcup_{\nu = 0}^\eta 
        \big((\eta-\nu)d+ \nu B_{s,\ell-1}\big). 
    \end{equation} Since $d > h(\ell-2)$, the sets on the right side of (\ref{eqtempor}) are disjoint, so $|\eta B_{1,\ell}| = \sum_{\nu = 0}^\eta|\nu B_{s,\ell-1}|$. Plugging in that $(1-z)^{-1} = (1+z+z^2+\cdots)$ yields (d).
\end{proof}

Now, we define a set $S_{m,\ell}$ to be a geometric series which includes 0, has common ratio $m$, and has $\ell$ elements. We then bound its generating function $\calf_{S_{m,\ell}}(z)$.

\begin{definition}
    For $m\geq 2$ and $\ell \geq 2$, define the set $S_{m,\ell}\coloneq  \{0,1,m,m^2,\ldots,m^{\ell-2}\}$.
\end{definition}

\begin{lemma}\label{lemmaHSML}
    For $m\geq 2$ and $\ell \geq 2$, we have 
    \[\calf_{S_{m,\ell}}(z) = \frac{1-(\ell-2)z^mO[[\ell z]]}{(1-z)^\ell}.\]
\end{lemma}

\begin{proof}
    Let $\eta \in \N$. The elements in $\eta S_{m,\ell}$ can be viewed as numbers $n$ where $\floor{\frac{n}{m^{\ell-2}}}$ plus the sum of the base-$m$ digits of $n \pmod{m^{\ell-2}}$ is less than or equal to $\eta$. To construct such an $n$, construct $n$ with $\floor{\frac{n}{m^{\ell-2}}}$ copies of $m^{\ell-2}$, and then use the base expansion of $n \pmod{m^{\ell-2}}$ in base $m$. This is a general construction of all elements $n$ in $\eta S_{m,\ell}$. 

    This gives an expression of $|\eta S_{m,\ell}|$ as the number of ways to put $\eta$ balls in $\ell$ boxes, where $\ell-2$ of the boxes have a maximum number of $m-1$ balls, and $2$ of the boxes have no maximum. The $\ell-2$ boxes correspond to the number of $1$'s, $m$'s, $\ldots$, $m^{\ell-3}$'s in the base-$m$ expansion of $n$. These numbers are bounded between $0$ and $m-1$. One infinite-capacity box corresponds to the number of $m^{\ell-2}$'s in $n$, and the other corresponds to the `overflow', or the number of $0$'s in $n$. In total, $n$ must be formed by exactly $\eta $ balls, as it is a sum of $\eta $ elements of $S_{m,\ell}$. This gives a generating function formula 
    \[\calf_{S_{m,\ell}}(z) = (1+z+z^2+\cdots+z^{m-1})^{\ell-2}(1+z+z^2+\cdots)^2 = \frac{(1-z^m)^{\ell-2}}{(1-z)^\ell}.\]
    Expanding the numerator $(1-z^m)^{\ell-2}$ out using the binomial theorem yields the desired result.
\end{proof}

Now, we will define a similar set $T_{m,\ell}$, which is a geometric series with common ratio $m$ and $\ell$ elements which does not include 0. With a bit more work, we prove a similar result about the generating function $\calf_{T_{m,\ell}}(z)$. 

\begin{definition}
    For $m\geq 2$ and $\ell  \geq 2$, define the set $T_{m,\ell}\coloneq  \{1,m,m^2,\ldots,m^{\ell-1}\} = S_{m,\ell+1} \setminus \{0\}.$
\end{definition}

\begin{lemma}\label{lemmaHTML}
    For $m\geq 2$ and $\ell \geq 2$, we have 
    \[\calf_{T_{m,\ell}}(z) = \frac{1-\binom{\ell-1}{2}z^{m+1}O[[\ell z]]}{(1-z)^\ell}.\]
\end{lemma}

\begin{proof}
    Let $\eta \in \N$. We will approximate $|\eta T_{m,\ell}|$ by putting sums of elements of $T_{m,\ell}$ into a standardized form. Say that a representation $n = a_0 +a_1m + a_2m^2 + \cdots + a_{\ell-1}m^{\ell-1}$ is in $r$-\emph{normal form} if $a_0,\ldots,a_{\ell-2}$ are in $[0,m-1]$ and $a_{\ell-1} \geq 0$ and $a_0+\cdots + a_{\ell-1} = r$. By taking the base-$m$ expansion of $n$, every $n \in \N$ can be written in normal form uniquely.

    Each element $n \in \eta T_{m,\ell}$ can be written as $n = a_0 +a_1m + a_2m^2 + \cdots + a_{\ell-1}m^{\ell-1}$, where $a_i \geq 0$ and $a_0+\cdots + a_{\ell-1} = \eta $. This representation of $n$ might not be unique, so we will put $n$ into normal form. 
 If $a_i \geq m$ for some $i < \ell -1$, we can decrement $a_i$ by $m$ and add $1$ to $a_{i+1}$ without changing the value of $n$; after this procedure then $a_0+\cdots + a_{\ell-1} = \eta -(m-1)$. Iterating this whenever possible, we will eventually get that $a_0,\ldots,a_{\ell-2}$ are in $[0,m-1]$ and $a_{\ell-1}$ is any nonnegative integer; if we have done this $j$ times, we get the sum $a_0+\cdots + a_{\ell-1} = \eta -j(m-1)$. Thus, any element $n \in \eta T_{m,\ell}$ can be put into $\eta -j(m-1)$-normal form for some $j \geq 0$, and this form is unique.

    Now, we reverse this process. Consider an element $n$ in $\eta -j(m-1)$-normal form. To pull back $n$ into $\eta T_{m,\ell}$, we need to decrement $a_i$ by 1 and increment $a_{i-1}$ by $m$ for $j$ different nonzero values of $i$. This is possible as long as we never reach a state where $a_i = 0$ for all $i>0$, which would mean $n$ is represented as a sum of $\eta -j'(m-1)$ copies of 1. This will only happen when $n = \eta -j'(m-1)$ for $j' \geq 1$. This rules out $\floor{\frac{\eta }{m-1}}$ numbers. Therefore, $|\eta T_{m,\ell}|$ is equal to the number of elements $n$ in $\eta -j(m-1)$-normal form for any $j \geq 0$, minus $\floor{\frac{\eta }{m-1}}$.
    
    We now create the generating function $\calf_{T_{m,\ell}}$. Elements in $r$-normal form can be counted as the number of ways to put $r$ balls in $\ell$ boxes, where the number of balls in box $i$ corresponds to $a_i$. There are $\ell-1$ boxes which have a maximum of $m-1$ balls, and $1$ box with no maximum number of balls. Then, multiplying by $(1+z^{m-1}+z^{2(m-1)}+\cdots)$ to account for elements in $\eta -j(m-1)$-normal form, and subtracting $\floor{\frac{\eta }{m-1}}$, gives the generating function:

   \begin{align*}
        \calf_{T_{m,\ell}}(z) &= (1+z+z^2+\cdots+z^{m-1})^{\ell-1}(1+z+z^2+\cdots)(1+z^{m-1}+z^{2(m-1)}+\cdots) - \sum_{\eta =0}^\infty \floor*{\frac{\eta }{m-1}} z^\eta \\&= \frac{(1-z^m)^{\ell-1}}{(1-z)^\ell(1-z^{m-1})}- \sum_{\eta =0}^\infty \floor*{\frac{\eta }{m-1}} z^\eta .
    \end{align*}
    We can then simplify \[\sum_{\eta =0}^\infty \floor*{\frac{\eta }{m-1}} z^\eta  = z^{m-1}(1+z+z^2+\cdots)(1+z^{m-1}+z^{2(m-1)} + \cdots) = \frac{z^{m-1}}{(1-z)(1-z^{m-1})}.\] Plugging this into our formula for $\calf_{T_{m,\ell}}(z)$ gives \[\calf_{T_{m,\ell}}(z) = \frac{(1-z^m)^{\ell-1}-z^{m-1}(1-z)^{\ell-1}}{(1-z)^\ell(1-z^{m-1})} = \frac{1-z^{m-1} -\binom{\ell-1}{2}z^{m+1} +\sum_{i=3}^{\infty}O(\ell^i)z^{i+m-1}}{(1-z)^\ell(1-z^{m-1})}.\] Expanding the left side of this out for $\ell = 2$ and the right side out for $\ell > 2$ yields the desired result.
\end{proof}

\subsubsection{Putting it All Together}\label{sssection: finish}
We begin by defining a notion of a family of sets being connected.
\begin{definition}
    A set of sets $\cala = \{A_1,\ldots,A_r\}$ with all $|A_i| = k$ is \emph{connected} if the set of sumset sizes $h\cala \coloneq  \{|hA|:A \in \cala\}$ forms an interval.
\end{definition}
Now, using all of our work in Section \ref{sssection: components} to understand the component sets, we prove that the set of all $A$, with $b$ fixed and $j,s_m,t_{m-1}$ varying for $m \in [3,h]$, is connected.
\begin{lemma}\label{lemmaAb}
    Fix $h \geq 4$. Fix $k$ to be sufficiently large, fix $b \in [3,k-\floor{k^{0.7}}]$, and let $c = k-b$. Let $B_j\coloneq B_{j,b}$. Define the family of sets \[\cala_b = \left\{A = B_j \sqcup \bigsqcup_{m=3}^h S_{m,s_m} \sqcup \bigsqcup_{m=2}^h T_{m,t_m}:\substack{j \in [1,(h-1)(b-2)+1] \\ s_m \in [2,\floor{c^{0.9}}] \textnormal{ if }m \in [3,h] \\ t_{m-1} \in [2,\floor{c^{0.8}}]\textnormal{ if } m \in [3,h] \\
    t_h = c-\sum_{m=3}^h(s_m+t_{m-1})} 
    \right\}.\] Then $\cala_b$ is connected.
\end{lemma}
\begin{proof}
    We will apply Lemma \ref{lemmahypercubeIVT} with $d = 2h-3$ and \[(x_1,x_2,x_3,\ldots,x_{d}) = (s_h-1,t_{h-1}-1,s_{h-1}-1,t_{h-2}-1,\ldots,s_{3}-1,t_2-1,j).\] The values $n_i$ are given in (\ref{eqvalni}). So it will suffice to show that $|hA|$ is nonincreasing in each entry $x_i$, that $\delta_1 = 1$, and that $\delta_i \leq \Delta_{i-1}$ for $i \in [2,d]$.

    Let $C = \displaystyle\bigsqcup_{m=3}^h S_{m,s_m} \sqcup \displaystyle\bigsqcup_{m=2}^h T_{m,t_m}$. By Lemmas \ref{lemmaHSML} and \ref{lemmaHTML}, we have that \begin{equation}\label{eq201}
        \calf_{C}(z) = \frac{1}{(1-z)^c} \prod_{m=3}^h (1-(s_m-2)z^mO[[s_mz]])\cdot  \prod_{m=2}^h \left(1-\binom{t_m-1}{2}z^{m+1}O[[t_mz]]\right).
    \end{equation} By the binomial theorem, \begin{equation}\label{eqbinomial}\frac{1}{(1-z)^c}= \sum_{\eta=0}^\infty \binom{\eta+c-1}{\eta} z^\eta \in \Theta_+[[c z]].\end{equation}
    
    We now prove the following two assertions about how changing one $s_\mu$ or $t_\mu$ while keeping all of the other $s_m$ and $t_{m-1}$ unchanged for $m \in [3,h-1]$ will affect $\calf_{C}(z)$.
    \begin{enumerate}[label = (\arabic*)]
        \item[($\star$)] For $\mu \in [3,h]$, increasing $s_\mu$ by 1 decreases $\calf_C(z)$ by an element of $z^\mu\Theta_+[[cz]]$.
        \item[($\star \star$)]  For $\mu \in [2,h-1]$, increasing $t_\mu$ by 1 decreases $\calf_C(z)$ by an element of $(t_\mu-1)z^{\mu+1}\Theta_+[[cz]]$.
    \end{enumerate}
    Notice that we can ignore the factor coming from $t_h$ in (\ref{eq201}) since it only affects powers of $z$ greater than or equal to $h+1$, and also notice that each of the other factors in the products on the right side of (\ref{eq201}) is in $O[[c^{0.9}z]]$. 
    To prove ($\star$) and ($\star \star$), we now use Lemma \ref{lemmaProductsofGenFuncs} to combine all of the factors in the products on the right side of (\ref{eq201}). For ($\star$), increasing $s_\mu$ by 1 to $s_\mu'=s_\mu+1$ will decrease $\calf_C(z)$ by an element of \[\Theta_+[[cz]]((s_\mu'-2)-(s_\mu-2))z^\mu O[[s_\mu z]]\prod_{m\in [3,h] \setminus \{\mu\}} O[[c^{0.9}z]] \cdot \prod_{m=2}^{h-1} O[[c^{0.9}z]] = z^\mu\Theta_+[[cz]].\] For ($\star \star$), increasing $t_\mu$ by 1 to $t_\mu'=t_\mu+1$ decreases $\calf_C(z)$ by an element of 
    \[\Theta_+[[cz]]\left(\hspace{-2 pt}\binom{t_\mu'-1}{2}-\binom{t_\mu-1}{2}\hspace{-2 pt}\right)z^{\mu+1}O[[t_\mu z]] \prod_{m=3}^h O[[c^{0.9}z]] \cdot \hspace{-4.5 pt}\prod_{m \in [2,h-1] \setminus \{\mu\}}\hspace{-4.5 pt} O[[c^{0.9}z]]= (t_\mu-1)z^{\mu+1} \Theta_+[[cz]].\]

    Since $A = B_j \sqcup C$, we have \begin{equation}\label{eqhA}|hA| = \sum_{\eta = 0}^h |\eta C||(h-\eta)B_j|.\end{equation} By ($\star$) and ($\star \star$) and Lemma \ref{lemmadense}(a), we have that $|hA|$ is nonincreasing in each entry $x_i$. From ($\star$) and (\ref{eqhA}), it is easy to check that the only term affected by increasing $s_h$ by 1 will be $|hC|$, which will decrease by 1. Hence, $\delta_1 = 1$.
    
    We now approximate $\delta_i$ for $i \in [1,d]$ and $\Delta_{i}$ for $i \in [1,d-1]$, which we will use to prove that $\delta_i \leq \Delta_{i-1}$ for $i \in [2,d]$, which implies the lemma. Throughout, we use Lemma \ref{lemmadense}(b) to show that the terms with $\eta = h$ and $\eta = h-1$ dominate in (\ref{eqhA}). 

    \textbf{Case 1: $i$ is odd and $i < d$.} Let $x_i = s_\mu-1$ for $\mu \in [3,h]$. By ($\star$), increasing $s_\mu$ by 1 will decrease $\calf_C(z)$ by an element of $z^\mu\Theta_+[[cz]]$. By (\ref{eqhA}), \begin{equation}\label{eq501}
        \delta_i = \sum_{\eta = \mu}^h \Theta_+(c^{\eta-\mu})|(h-\eta)B_j| = \Theta_+(c^{h-\mu})+b\Theta_+(c^{h-\mu-1}).
    \end{equation} So, increasing $s_\mu$ from $2$ to $\floor{c^{0.9}}$ gives that \begin{equation}\label{eq502}\Delta_i = \left(\floor{c^{0.9}}-2\right)\sum_{\eta = \mu}^h \Theta_+(c^{\eta-\mu})|(h-\eta)B_j| = \Theta_+(c^{h-\mu+0.9})+b\Theta_+(c^{h-\mu-0.1}).\end{equation}

    \textbf{Case 2: $i$ is even.} Let $x_i = t_\mu-1$ for $\mu \in [2,h-1]$. By ($\star \star$), increasing $t_\mu$ by 1 will decrease $\calf_C(z)$ by an element of $(t_\mu-1)z^{\mu+1}\Theta_+[[cz]]$. By (\ref{eqhA}), \begin{equation}\label{eq503}\delta_{i} = \left(\floor{c^{0.8}}-1\right)\sum_{\eta = \mu+1}^h \Theta_+(c^{\eta-\mu-1})|(h-\eta)B_j| = \Theta_+(c^{h-\mu-0.2})+ b\Theta_+(c^{h-\mu-1.2}).\end{equation} By ($\star \star$), increasing $t_\mu$ from $2$ to $\floor{c^{0.8}}$ will decrease $\calf_C(z)$ by an element of $\Theta_+(c^{1.6})z^{\mu+1}\Theta_+[[cz]]$. By (\ref{eqhA}), \begin{equation}\label{eq504}
        \Delta_{i} = \sum_{\eta = \mu+1}^h \Theta_+(c^{\eta-\mu-1+1.6})|(h-\eta)B_j| \geq \Theta_+(c^{h-\mu+0.6}).
    \end{equation}

    \textbf{Case 3: $i = d$.} Here, $x_d = j$. By (\ref{eqhA}) and Lemma \ref{lemmadense}(a), \begin{equation}\label{eq505}
        \delta_d = \max_{j \in [1,(h-1)(b-2)]} \sum_{\eta = 0}^{h-2} |\eta C|\big(|(h-\eta)B_{j+1}|-|(h-\eta)B_{j}|\big) \leq h \sum_{\eta = 0}^{h-2} |\eta C| \leq \Theta_+(c^{h-2}).
    \end{equation}

    Now, we can compare the formulas of (\ref{eq501}), (\ref{eq502}), (\ref{eq503}), (\ref{eq504}), and (\ref{eq505}) to show $\delta_i \leq \Delta_{i-1}$ for $i \in [2,d]$. With $i$ being odd and less than $d$, let $x_i = s_{\nu}-1$ and $x_{i-1} = t_{\nu}-1$. Then we are comparing (\ref{eq501}) with $\mu = \nu$ to (\ref{eq504}) with $\mu = \nu$. Since $3 \leq b \leq k-\floor{k^{0.7}}$, and $c = k-b$, we have $c^{1.5} > b$. Hence, $\delta_i \leq \Delta_{i-1}$. With $i$ being even, let $x_i = t_{\nu}-1$ and $x_{i-1} = s_{\nu+1}-1$. Then we are comparing (\ref{eq503}) with $\mu=\nu$ to (\ref{eq502}) with $\mu = \nu+1$, in which case it is clear that $ \delta_i \leq \Delta_{i-1}$. With $i = d$, we are comparing (\ref{eq505}) to (\ref{eq504}) with $\mu = 2$, so it is clear that $ \delta_d \leq \Delta_{d-1}$. 

    So, we have checked all of the assumptions of Lemma \ref{lemmahypercubeIVT}. Applying Lemma \ref{lemmahypercubeIVT} yields the result.
\end{proof}

We are now ready to prove Proposition \ref{lemmaMain} by showing that the values of $|hA|$ for $A \in \cala_b$ with $b$ varying form an interval, and this interval contains a value smaller than $\varepsilon k^h$ and also contains $\binom{h+k-1}{h}$.

\begin{proof}[Proof of Proposition~\ref{lemmaMain}]
Let $\cala_b$ be as in Lemma~\ref{lemmaAb} and recall that $c = k-b$. Define \[\cala \coloneq  \bigcup_{b=3}^{k-\floor{k^{0.7}}} \cala_b.\] By Lemma~\ref{lemmaAb}, $\cala_b$ is connected. To show that $\cala$ is connected, it suffices to show that $h\cala_{b} \cap h\cala_{b-1} \neq \emptyset$ for $b \in [4,k-\floor{k^{0.7}}]$. 

For any $b \in [3,k-\floor{k^{0.7}}]$, define a representative set $A_b \in \cala_b$ by taking $(x_1,\ldots,x_d) = (1,\ldots,1,1)$ using (\ref{eqdefofA}) and (\ref{eqxi}). For any $b \in [4,k-\floor{k^{0.7}}]$, define another representative set $A_{b-1}' \in \cala_{b-1}$ by taking $(x_1,\ldots,x_d) = (1,\ldots,1,n_d)$ using (\ref{eqdefofA}) and (\ref{eqxi}). Let $A_b = B_{1,b} \sqcup C$ and $A_{b-1}' = B_{s,b-1} \sqcup C'$. As ${s_\mu = t_{\mu-1} = 2}$ for all $\mu \in [3,h]$ for $C$, (\ref{eq201}) gives that $\calf_C(z) \equiv (1-z)^{-|C|} = (1-z)^{-(k-b)}$. Similarly, $\calf_{C'}(z) \equiv (1-z)^{-|C'|} = (1-z)^{-(k-b+1)}$. By Lemma \ref{lemmadense}(d), we have $\calf_{B_{1,b}}(z) \equiv (1-z)^{-1} \calf_{B_{s,b-1}}(z)$. Hence, \[\calf_{A_b}(z) = \calf_{B_{1,b}}(z) \calf_C(z) \equiv  \calf_{B_{s,b-1}}(z)(1-z)^{-1}(1-z)^{-(k-b)} \equiv \calf_{B_{s,b-1}}(z) \calf_{C'}(z)= \calf_{A_{b-1}'}(z),\] so $|hA_b| = |hA_{b-1}'|$. So $h\cala_{b} \cap h\cala_{b-1} \neq \emptyset$.

We now show that $\binom{h+k-1}{h} \in h\cala_3$ and that there is a value smaller than $\varepsilon k^h$ in $h\cala_{k-\floor{k^{0.7}}}$. This will imply that the interval $h\cala$ contains the desired interval. Let $A_3 = B_{1,3} \sqcup C$, so by similar logic to the previous paragraph we get $\calf_C(z) \equiv (1-z)^{-(k-3)}$. Plugging into Lemma \ref{lemmadense}(c) gives the formula \[\calf_{A_3}(z) \equiv \calf_{B_{1,3}}(z) \calf_C(z) \equiv (1-z)^{-3}(1-z)^{-(k-3)} \equiv (1-z)^{-k}.\] By the binomial theorem, $|hA_3| = \binom{h+k-1}{h} \in h\cala_3$. 

In a similar fashion, we get that \begin{equation}\label{eq301}
\calf_{A_{k-\floor{k^{0.7}}}}(z) \equiv (1-z)^{-\floor{k^{0.7}}}\calf_{B_{1,k-\floor{k^{0.7}}}}(z).\end{equation} By Lemma~\ref{lemmadense}(b), $|\eta B_{1,{k-\floor{k^{0.7}}}}| \in O(k)$ for all $1 \leq \eta \leq h$. By (\ref{eqbinomial}), $(1-z)^{-\floor{k^{0.7}}} \in \Theta_+[[k^{0.7}z]]$. Multiplying out the $z^h$ term of (\ref{eq301}) gives that $|hA_{k-\floor{k^{0.7}}}| \in O(k^{0.7(h-1)+1})$. Since $0.7(h-1)+1 < h$, we can get $|hA_{k-\floor{k^{0.7}}}| < \varepsilon k^h$ for any desired $\varepsilon >0$ by making $k$ large enough.
\end{proof}

\section{Future Directions}\label{sectfuturedirections}

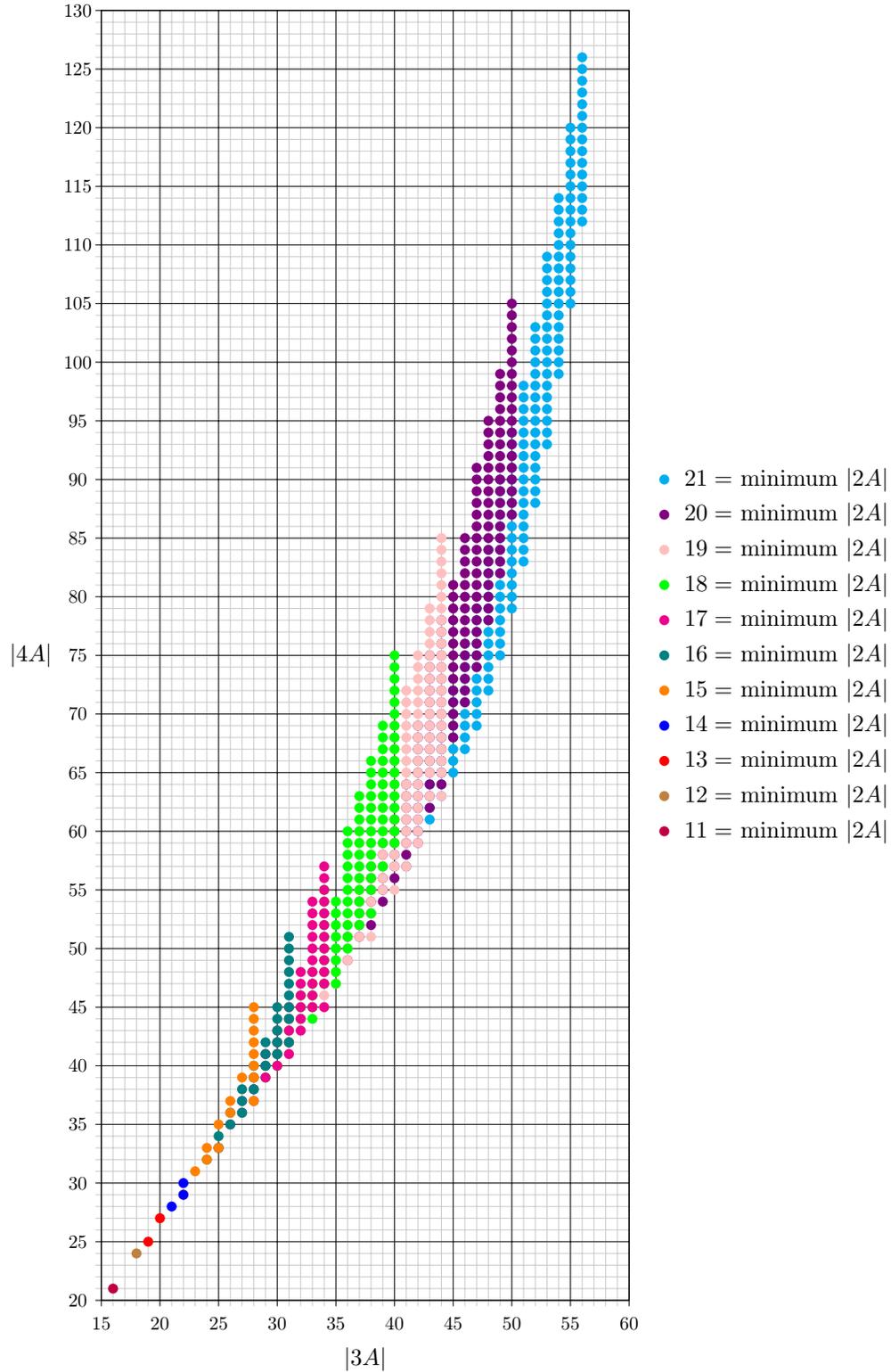
\begin{figure}[!ht]
    \centering
\begin{tikzpicture}[scale=0.153]
\draw[step=1, black!20, thin] (60,130) grid ($(15,20)-(0.5,0.5)$);
\draw[step=5, black!90] (60,130) grid ($(15,20)-(0.5,0.5)$);

\draw (60,130) rectangle (15,20);

\foreach \x in {20,25,...,130} {
    \node[scale=0.8] at (13,\x) {\x};
}
\foreach \x in {15,20,...,60} {
    \node[scale=0.8] at (\x,18) {\x};
}

\node at (37.5,15) {$|3A|$};

\node at (9,75) {$|4A|$};

\foreach \x in {
{(42,60)},{(42,62)},{(43,61)},{(43,62)},{(43,64)},{(43,66)},{(44,63)},{(44,64)},{(44,65)},{(44,66)},{(44,67)},{(44,68)},{(45,65)},{(45,66)},{(45,67)},{(45,68)},{(45,69)},{(45,70)},{(45,71)},{(45,72)},{(46,67)},{(46,68)},{(46,69)},{(46,70)},{(46,71)},{(46,72)},{(46,73)},{(46,74)},{(46,75)},{(46,76)},{(47,69)},{(47,70)},{(47,71)},{(47,72)},{(47,73)},{(47,74)},{(47,75)},{(47,76)},{(47,77)},{(47,78)},{(47,79)},{(48,72)},{(48,73)},{(48,74)},{(48,75)},{(48,76)},{(48,77)},{(48,78)},{(48,79)},{(48,80)},{(48,81)},{(48,82)},{(48,83)},{(48,84)},{(49,75)},{(49,76)},{(49,77)},{(49,78)},{(49,79)},{(49,80)},{(49,81)},{(49,82)},{(49,83)},{(49,84)},{(49,85)},{(49,86)},{(49,87)},{(49,88)},{(49,89)},{(49,90)},{(50,79)},{(50,80)},{(50,81)},{(50,82)},{(50,83)},{(50,84)},{(50,85)},{(50,86)},{(50,87)},{(50,88)},{(50,89)},{(50,90)},{(50,91)},{(50,92)},{(50,93)},{(50,94)},{(51,83)},{(51,84)},{(51,85)},{(51,86)},{(51,87)},{(51,88)},{(51,89)},{(51,90)},{(51,91)},{(51,92)},{(51,93)},{(51,94)},{(51,95)},{(51,96)},{(51,97)},{(51,98)},{(52,88)},{(52,89)},{(52,90)},{(52,91)},{(52,92)},{(52,93)},{(52,94)},{(52,95)},{(52,96)},{(52,97)},{(52,98)},{(52,99)},{(52,100)},{(52,101)},{(52,102)},{(52,103)},{(53,93)},{(53,94)},{(53,95)},{(53,96)},{(53,97)},{(53,98)},{(53,99)},{(53,100)},{(53,101)},{(53,102)},{(53,103)},{(53,104)},{(53,105)},{(53,106)},{(53,107)},{(53,108)},{(53,109)},{(54,99)},{(54,100)},{(54,101)},{(54,102)},{(54,103)},{(54,104)},{(54,105)},{(54,106)},{(54,107)},{(54,108)},{(54,109)},{(54,110)},{(54,111)},{(54,112)},{(54,113)},{(54,114)},{(55,105)},{(55,106)},{(55,107)},{(55,108)},{(55,109)},{(55,110)},{(55,111)},{(55,112)},{(55,113)},{(55,114)},{(55,115)},{(55,116)},{(55,117)},{(55,118)},{(55,119)},{(55,120)},{(56,112)},{(56,113)},{(56,114)},{(56,115)},{(56,116)},{(56,117)},{(56,118)},{(56,119)},{(56,120)},{(56,121)},{(56,122)},{(56,123)},{(56,124)},{(56,125)},{(56,126)}}{
    \node[cyan, vtx] (cyan\x) at \x {};
}
\foreach \x in {
{(36,49)},{(36,51)},{(37,51)},{(37,52)},{(37,53)},{(37,54)},{(38,52)},{(38,53)},{(38,54)},{(38,55)},{(38,56)},{(38,57)},{(39,54)},{(39,55)},{(39,56)},{(39,57)},{(39,58)},{(39,59)},{(40,56)},{(40,57)},{(40,58)},{(40,59)},{(40,60)},{(40,61)},{(41,57)},{(41,58)},{(41,59)},{(41,60)},{(41,61)},{(41,62)},{(41,63)},{(41,64)},{(42,59)},{(42,60)},{(42,61)},{(42,62)},{(42,63)},{(42,64)},{(42,65)},{(42,66)},{(42,67)},{(42,68)},{(42,69)},{(43,62)},{(43,63)},{(43,64)},{(43,65)},{(43,66)},{(43,67)},{(43,68)},{(43,69)},{(43,70)},{(43,71)},{(43,72)},{(43,73)},{(43,74)},{(43,75)},{(44,64)},{(44,65)},{(44,66)},{(44,67)},{(44,68)},{(44,69)},{(44,70)},{(44,71)},{(44,72)},{(44,73)},{(44,74)},{(44,75)},{(44,76)},{(44,77)},{(44,78)},{(45,68)},{(45,69)},{(45,70)},{(45,71)},{(45,72)},{(45,73)},{(45,74)},{(45,75)},{(45,76)},{(45,77)},{(45,78)},{(45,79)},{(45,80)},{(45,81)},{(46,71)},{(46,72)},{(46,73)},{(46,74)},{(46,75)},{(46,76)},{(46,77)},{(46,78)},{(46,79)},{(46,80)},{(46,81)},{(46,82)},{(46,83)},{(46,84)},{(46,85)},{(47,74)},{(47,75)},{(47,76)},{(47,77)},{(47,78)},{(47,79)},{(47,80)},{(47,81)},{(47,82)},{(47,83)},{(47,84)},{(47,85)},{(47,86)},{(47,87)},{(47,88)},{(47,89)},{(47,90)},{(47,91)},{(48,78)},{(48,79)},{(48,80)},{(48,81)},{(48,82)},{(48,83)},{(48,84)},{(48,85)},{(48,86)},{(48,87)},{(48,88)},{(48,89)},{(48,90)},{(48,91)},{(48,92)},{(48,93)},{(48,94)},{(48,95)},{(49,82)},{(49,83)},{(49,84)},{(49,85)},{(49,86)},{(49,87)},{(49,88)},{(49,89)},{(49,90)},{(49,91)},{(49,92)},{(49,93)},{(49,94)},{(49,95)},{(49,96)},{(49,97)},{(49,98)},{(49,99)},{(50,87)},{(50,88)},{(50,89)},{(50,90)},{(50,91)},{(50,92)},{(50,93)},{(50,94)},{(50,95)},{(50,96)},{(50,97)},{(50,98)},{(50,99)},{(50,100)},{(50,101)},{(50,102)},{(50,103)},{(50,104)},{(50,105)}}{
    \node[violet, vtx] (violet\x) at \x {};
}
\foreach \x in {
{(29,39)},{(30,41)},{(31,41)},{(31,42)},{(32,43)},{(32,44)},{(32,45)},{(33,45)},{(33,46)},{(33,47)},{(34,45)},{(34,46)},{(34,47)},{(34,48)},{(34,49)},{(35,47)},{(35,48)},{(35,49)},{(35,50)},{(35,51)},{(36,49)},{(36,50)},{(36,51)},{(36,52)},{(36,53)},{(37,51)},{(37,52)},{(37,53)},{(37,54)},{(37,55)},{(37,56)},{(37,57)},{(37,58)},{(37,59)},{(37,60)},{(37,61)},{(38,51)},{(38,53)},{(38,54)},{(38,55)},{(38,56)},{(38,57)},{(38,58)},{(38,59)},{(38,60)},{(38,61)},{(38,62)},{(38,63)},{(39,55)},{(39,56)},{(39,57)},{(39,58)},{(39,59)},{(39,60)},{(39,61)},{(39,62)},{(39,63)},{(39,64)},{(39,65)},{(39,66)},{(40,55)},{(40,57)},{(40,58)},{(40,59)},{(40,60)},{(40,61)},{(40,62)},{(40,63)},{(40,64)},{(40,65)},{(40,66)},{(40,67)},{(40,68)},{(40,69)},{(41,57)},{(41,59)},{(41,60)},{(41,61)},{(41,62)},{(41,63)},{(41,64)},{(41,65)},{(41,66)},{(41,67)},{(41,68)},{(41,69)},{(41,70)},{(41,71)},{(41,72)},{(42,59)},{(42,60)},{(42,61)},{(42,62)},{(42,63)},{(42,64)},{(42,65)},{(42,66)},{(42,67)},{(42,68)},{(42,69)},{(42,70)},{(42,71)},{(42,72)},{(42,73)},{(42,74)},{(42,75)},{(43,63)},{(43,65)},{(43,66)},{(43,67)},{(43,68)},{(43,69)},{(43,70)},{(43,71)},{(43,72)},{(43,73)},{(43,74)},{(43,75)},{(43,76)},{(43,77)},{(43,78)},{(43,79)},{(44,63)},{(44,65)},{(44,66)},{(44,67)},{(44,68)},{(44,69)},{(44,70)},{(44,71)},{(44,72)},{(44,73)},{(44,74)},{(44,75)},{(44,76)},{(44,77)},{(44,78)},{(44,79)},{(44,80)},{(44,81)},{(44,82)},{(44,83)},{(44,84)},{(44,85)}}{
    \node[pink, vtx] (pink\x) at \x {};
}
\foreach \x in {
{(28,38)},{(29,39)},{(29,40)},{(30,40)},{(30,41)},{(30,42)},{(31,42)},{(31,43)},{(31,44)},{(32,44)},{(32,45)},{(32,46)},{(33,44)},{(33,45)},{(33,46)},{(33,47)},{(33,48)},{(34,47)},{(34,48)},{(34,49)},{(34,50)},{(34,51)},{(34,52)},{(34,53)},{(34,54)},{(34,55)},{(35,47)},{(35,48)},{(35,49)},{(35,50)},{(35,51)},{(35,52)},{(35,53)},{(35,54)},{(36,50)},{(36,51)},{(36,52)},{(36,53)},{(36,54)},{(36,55)},{(36,56)},{(36,57)},{(36,58)},{(36,59)},{(36,60)},{(37,52)},{(37,53)},{(37,54)},{(37,55)},{(37,56)},{(37,57)},{(37,58)},{(37,59)},{(37,60)},{(37,61)},{(37,62)},{(37,63)},{(38,53)},{(38,55)},{(38,56)},{(38,57)},{(38,58)},{(38,59)},{(38,60)},{(38,61)},{(38,62)},{(38,63)},{(38,64)},{(38,65)},{(38,66)},{(39,57)},{(39,59)},{(39,60)},{(39,61)},{(39,62)},{(39,63)},{(39,64)},{(39,65)},{(39,66)},{(39,67)},{(39,68)},{(39,69)},{(40,59)},{(40,60)},{(40,61)},{(40,62)},{(40,63)},{(40,64)},{(40,65)},{(40,66)},{(40,67)},{(40,68)},{(40,69)},{(40,70)},{(40,71)},{(40,72)},{(40,73)},{(40,74)},{(40,75)}}{
    \node[green, vtx] (green\x) at \x {};
}
\foreach \x in {
{(25,33)},{(26,35)},{(27,36)},{(27,37)},{(28,37)},{(28,38)},{(28,39)},{(29,39)},{(29,40)},{(29,41)},{(30,40)},{(30,41)},{(30,42)},{(30,43)},{(31,41)},{(31,42)},{(31,43)},{(31,44)},{(31,45)},{(32,43)},{(32,44)},{(32,45)},{(32,46)},{(32,47)},{(32,48)},{(33,45)},{(33,46)},{(33,47)},{(33,48)},{(33,49)},{(33,50)},{(33,51)},{(33,52)},{(33,53)},{(33,54)},{(34,45)},{(34,47)},{(34,48)},{(34,49)},{(34,50)},{(34,51)},{(34,52)},{(34,53)},{(34,54)},{(34,55)},{(34,56)},{(34,57)}}{
    \node[magenta, vtx] (magenta\x) at \x {};
}
\foreach \x in {
{(24,32)},{(25,33)},{(25,34)},{(26,35)},{(26,36)},{(27,36)},{(27,37)},{(27,38)},{(28,38)},{(28,39)},{(28,40)},{(29,40)},{(29,41)},{(29,42)},{(30,41)},{(30,42)},{(30,43)},{(30,44)},{(30,45)},{(31,42)},{(31,44)},{(31,45)},{(31,46)},{(31,47)},{(31,48)},{(31,49)},{(31,50)},{(31,51)}}{
    \node[teal, vtx] (teal\x) at \x {};
}
\foreach \x in {
{(22,29)},{(23,31)},{(24,32)},{(24,33)},{(25,33)},{(25,35)},{(26,36)},{(26,37)},{(27,39)},{(28,37)},{(28,39)},{(28,40)},{(28,41)},{(28,42)},{(28,43)},{(28,44)},{(28,45)}}{
    \node[orange, vtx] (orange\x) at \x {};
}
\foreach \x in {
{(21,28)},{(22,29)},{(22,30)}}{
    \node[blue, vtx] (blue\x) at \x {};
}
\foreach \x in {
{(19,25)},{(20,27)}}{
    \node[red, vtx] (red\x) at \x {};
}
\foreach \x in {
{(18,24)}}{
    \node[brown, vtx] (brown\x) at \x {};
}
\foreach \x in {
{(16,21)}}{
    \node[purple, vtx] (purple\x) at \x {};
}

\node[cyan, vtx] at (63,95) {};
\node[right] at (64,95) {$21 = $ minimum $|2A|$};
\node[violet, vtx] at (63,91) {};
\node[right] at (64,91) {$20 = $ minimum $|2A|$};
\node[pink, vtx] at (63,87) {};
\node[right] at (64,87) {$19 = $ minimum $|2A|$};
\node[green, vtx] at (63,83) {};
\node[right] at (64,83) {$18 = $ minimum $|2A|$};
\node[magenta, vtx] at (63,79) {};
\node[right] at (64,79) {$17 = $ minimum $|2A|$};
\node[teal, vtx] at (63,75) {};
\node[right] at (64,75) {$16 = $ minimum $|2A|$};
\node[orange, vtx] at (63,71) {};
\node[right] at (64,71) {$15 = $ minimum $|2A|$};
\node[blue, vtx] at (63,67) {};
\node[right] at (64,67) {$14 = $ minimum $|2A|$};
\node[red, vtx] at (63,63) {};
\node[right] at (64,63) {$13 = $ minimum $|2A|$};
\node[brown, vtx] at (63,59) {};
\node[right] at (64,59) {$12 = $ minimum $|2A|$};
\node[purple, vtx] at (63,55) {};
\node[right] at (64,55) {$11 = $ minimum $|2A|$};

\end{tikzpicture}
\caption{The points $(|3A|,|4A|)$ for all $A \subset \Z$ with $|A| = 6$. The colors correspond to the minimum value of $|2A|$ for specified $|3A|$ and $|4A|$.}\label{fig3}
\end{figure}

\subsection{Working towards Conjecture~\ref{mainconj}}
Let $k_h'$ be the \emph{minimal} natural number such that for all $k>k_h'$, we have \begin{equation}\label{eqconj}\calr(h,k) = \left[hk-h+1,\binom{h+k-1}{h}\right]\setminus \Delta_{h,k}.\end{equation} Then, Theorem~\ref{thmgeneralh} proves that $k_h'$ exists, and by Remark \ref{remarkasterisk}, $k_h' \leq 2^{20h^2}$. Conjecture~\ref{mainconj} is that $k_h' \leq h$. It would be interesting to try to improve the upper bound on $k_h'$ to exponential in $h$, or polynomial in $h$, in order to work towards Conjecture~\ref{mainconj}. Or, if these bounds were improved for small $h$, it might be possible to explicitly compute that $k_4' = 4$ or $k_5' = 4$, which is what numerical data suggests. We pose the question of whether the set of all $k$ satisfying (\ref{eqconj}) is always a half-open interval given by $[k_{h}'+1,\infty)$.

Another natural follow-up would be to extend the inductive arguments from the proof of Theorem~\ref{thmh3} to $h=4$ or $h=5$. This could be approached using geometric pictures like Figure~\ref{fig3}. We have been unable to do this using the exact methods from Theorem~\ref{thmh3} because it seems necessary to control $|2A|$, $|3A|$, and $|4A|$ all at once. 

One could also try to compute more data to support or disprove Conjecture~\ref{mainconj} for small $h$ and $k$. So far, we have used a computer to show that this conjecture holds for all $(h,k)$ satisfying $h+k < 12$. Our algorithm uses a brute force method;  it enumerates all sets $A$ with $|A| = k$ and $\text{diam}(A) \leq d$ for some fixed $d$ and then computes $|hA|$ for all of them. More clever or efficient algorithms might be able to find much more data to support or disprove Conjecture~\ref{mainconj}.
\subsection{What if \texorpdfstring{$h \geq k$}{h > k}?}\label{secthgeqk}
Let a ``gap'' in $\calr(h,k)$ be a number in $\left[hk-h+1,\binom{h+k-1}{h}\right]\setminus \calr(h,k)$. For many pairs $(h,k)$ with $h \geq k$, there are gaps in $\calr(h,k)$ that are not in $\Delta_{h,k}$; in other words, (\ref{eqconj}) does not hold. We pose the problem of studying where all of these gaps are for certain values of $(h,k)$ and proving that many of these other gaps exist.

One place to start would be Theorem \ref{thmlev}, which is due to \cite{Lev96}. If $h > k+1$, similar methods to our proof of Theorem \ref{mainthm} with larger diameters yield other gaps in $\calr(h,k)$ outside $\Delta_{h,k}$. For example, if $\diam(A) \leq 2k-3$, then $|hA| \leq h(2k-3) + 1$ and if $\diam(A) \geq 2k-2$ then similar logic to (\ref{eqboundhA}) gives that $|hA| \geq k-1 + (h-1)(2k-2) = h(2k-3) + h-k+1$. Combining these bounds yields that $\calr(h,k) \cap [h(2k-3)+2,h(2k-3)+h-k] = \emptyset$. Similar analysis with larger values of $\diam(A)$ and larger values of $h-k$ will yield more gaps. 

With $k=3$ and $h \geq 4$, Nathanson \cite[Theorem 9]{nathanson25Problems} found the exact form of $\calr(h,k)$, showing there are many more gaps in $\calr(h,k)$. With $k=4$ and $h \geq 4$, experimental data suggests $5h \not \in \calr(h,4)$, which is an example of a conjectured gap. The only (conjectured) gaps in $\calr(h,k)$ outside $\Delta_{h,k}$ we have found with $h \leq k+1$ occur when $(h,k) \in \{(4,3),(4,4),(5,4)\}$. When $h$ is significantly larger than $k$, there are many more gaps in $\calr(h,k)$ than can be explained by Theorem \ref{thmlev}. One could conjecture where these other gaps occur, and try to prove that any of them exist. 

\subsection{Restricted Sumsets}
For $A \subset \Z$, define the \emph{$h$-fold restricted sumset of $A$} by \[\widehat{hA} \coloneq  \{a_1+\cdots+a_h: a_i \in A \text{ and } a_i \neq a_j \text{ if } i \neq j \}.\] Then, Nathanson defined $\widehat{\calr}(h,k)$ analogously to $\calr(h,k)$.
\begin{definition}\cite{nathanson25Problems}
      Define the \emph{range of cardinalities of $h$-fold restricted sumsets } \[\widehat{\calr}(h,k)\coloneq  \{|\widehat{hA}|:A \subset \Z \text{ and } |A| = k\}.\]
\end{definition}
All of the questions asked or answered in this paper about $\calr(h,k)$ can be asked about $\widehat{\calr}(h,k)$, and one could study whether similar results to Theorems \ref{mainthm}, \ref{thmgeneralh}, and \ref{thmh3} hold in this context.
\subsection{Working with Torsion}
The results in this paper also hold with $\Z$ replaced by any infinite torsion-free abelian group. We pose the problem of determining the form of $\calr(h,k)$ or $\widehat{\calr}(h,k)$ in groups with torsion. One interesting candidate for such a group is $(\Z/m\Z)^n$, with $m \in \N \setminus \{1\}$ and $n \in \N \cup \{\infty\}$. This seems interesting already when $m = 2$.

\subsection{Studying \texorpdfstring{$|hA|$}{|hA|} and \texorpdfstring{$|h'A|$}{|h'A|} together}
Another natural follow-up question is to try to figure out the exact form of the points in Figures \ref{fig2} and \ref{fig3}. In particular, one can ask what all of the possible tuples $(|hA|,|h'A|)$ are, where $A \subset \Z$ and $|A| = k$. With $h = 2$ and $h' = 3$, this is already a very difficult question, as illustrated in Figure~\ref{fig2}. The Plunnecke--Ruzsa \cite{Plunnecke70,Ruzsa89} inequality gives a bound on where these points are found, but it is not close to tight. One could try to bound the points in a plot such as Figure~\ref{fig2} or Figure~\ref{fig3} between two curves, in a tight or nearly tight way. There are also many other questions that can be asked about the structure of the points in these plots; for example, in Figure~\ref{fig3} we can ask about the possible tuples $(|3A|,|4A|)$ if we fix $|A|$ and $|2A|$, or about the possible tuples $(|2A|,|3A|,|4A|)$ if we fix $|A|$. All of these questions would be particularly interesting as $k$ gets large or in the limit as $k \to \infty$. 

In graph theory, Razborov \cite{Razborov08} proved an inequality about the set of pairs of possible edge and triangle densities in a graph. The illustration of this in \cite[Figure 5.1]{Zhao23} looks somewhat similar in character to Figure~\ref{fig2}. It would be fascinating to see if there is a relation between the sets of possible tuples $\big\{\big(|\{\text{edges in }G\}|, |\{\text{triangles in }G\}|\big):G \text{ a graph with $k$ vertices}\big\}$ and $\{(|2A|,|3A|):A \subset \Z \text{ and } |A| = k\}$. 

%%% AUTHOR: optional appendix here
%\appendix %% you may comment this out if no Appendix
%\section*{Appendix}
%\section{Improving the constants}
%Material is placed here as needed.

%%% AUTHOR: optional acknowledgments here
\section*{Acknowledgments}
This research was conducted at the University of Minnesota Duluth REU with support from Jane Street Capital, NSF Grant 2409861, and donations from Ray Sidney and Eric Wepsic. I thank Joe Gallian and Colin Defant for providing this wonderful opportunity. I thank Noah Kravitz for suggesting this project and advising my whole research process, during which he gave countless pieces of detailed and invaluable feedback and suggestions. I also thank Daniel Zhu and Carl Schildkraut for many helpful suggestions and pieces of feedback during my work on this project. I thank Eliot Hodges, Noah Kravitz, Mitchell Lee, Rupert Li, and Maya Sankar for advising the whole Duluth REU. I thank Dr. Mohan of the B.K. Birla Institute of Engineering and Technology for bringing the reference \cite{TX19} to my attention.

%%% AUTHOR:
%%% Bibliography goes here. Note that the arXiv cannot process bibtex
%%% or biber bibliographies.  Example of acceptable bibliograpy format:
\bibliographystyle{amsplain}

\begin{thebibliography}{99}

\bibitem[ES83]{ES83}
Paul Erd\H{o}s and Endre Szemer\'edi.
\newblock On sums and products of integers.
\newblock In {\em Studies in pure mathematics}, pages 213--218. Birkh\"auser, Basel, 1983.

\bibitem[GR26]{GR26}
Timothy Gowers and Isaac Rajagopal.
\newblock A recent experience with {C}hat{GPT} 5.5 {P}ro.
\newblock {\em Gowers's Weblog}, 2026.
\newblock https://gowers.wordpress.com/2026/05/08/a-recent-experience-with-chatgpt-5-5-pro/.

\bibitem[Lev96]{Lev96}
Vsevolod~F. Lev.
\newblock Structure theorem for multiple addition and the {F}robenius problem.
\newblock {\em J. Number Theory}, 58(1):79--88, 1996.

\bibitem[Nat96]{Nathanson96}
Melvyn~B. Nathanson.
\newblock {\em Additive number theory: Inverse problems and the geometry of sumsets}, volume 165 of {\em Graduate Texts in Mathematics}.
\newblock Springer-Verlag, New York, 1996.

\bibitem[Nat25a]{nathanson2025additivesumsetsizestetrahedral}
Melvyn~B. Nathanson.
\newblock {A}dditive sumset sizes with tetrahedral differences, 2025.
\newblock arXiv: 2507.08646.

\bibitem[Nat25b]{nathanson2025explicitsumsetsizesadditive}
Melvyn~B. Nathanson.
\newblock {E}xplicit sumset sizes in additive number theory.
\newblock {\em Canad. Math. Bull.}, 176:1--12, 2025.

\bibitem[Nat25c]{nathanson25Problems}
Melvyn~B. Nathanson.
\newblock Problems in additive number theory, {VI}: Sizes of sumsets.
\newblock {\em Acta Math. Hungar.}, 176:498--521, 2025.

\bibitem[Nat26a]{nathanson2025triangulartetrahedralnumberdifferences}
Melvyn~B. Nathanson.
\newblock {T}riangular and tetrahedral number differences of sumset sizes in additive number theory.
\newblock {\em Integers}, 26, 2026.

\bibitem[Nat26b]{nathanson2025compressioncomplexitysumsetsizes}
Melvyn~B. Nathanson.
\newblock Compression and complexity for sumset sizes in additive number theory.
\newblock {\em J. Number Theory}, 281:321--343, 2026.

\bibitem[O'B25]{obryant2025nathansonstriangularnumberphenomenon}
Kevin O'Bryant.
\newblock On {N}athanson's triangular number phenomenon, 2025.
\newblock arXiv: 2506.20836.

\bibitem[Pl{\"{u}}70]{Plunnecke70}
Helmut Pl{\"{u}}nnecke.
\newblock Eine zahlentheoretische anwendung der graphentheorie.
\newblock {\em J. Reine Angew. Math.}, 243:171--183, 1970.

\bibitem[Raz08]{Razborov08}
Alexander~A. Razborov.
\newblock On the minimal density of triangles in graphs.
\newblock {\em Combin. Probab. Comput.}, 17(4):603--618, 2008.

\bibitem[Ruz89]{Ruzsa89}
Imre~Z. Ruzsa.
\newblock An application of graph theory to additive number theory.
\newblock {\em Sci. Ser. A Math. Sci. (N.S.)}, 3:97--109, 1989.

\bibitem[Sch26]{Schinina25}
Vincent Schinina.
\newblock On the sumset of sets of size $k$.
\newblock {\em Integers}, 26, 2026.

\bibitem[TX21]{TX19}
Min Tang and Yun Xing.
\newblock Some inverse results of sumsets.
\newblock {\em Bull. Korean Math. Soc.}, 58(2):305--313, 2021.

\bibitem[Zha23]{Zhao23}
Yufei Zhao.
\newblock {\em Graph theory and additive combinatorics---exploring structure and randomness}.
\newblock Cambridge University Press, Cambridge, 2023.

\end{thebibliography}

%% AUTHOR: You can generate such a bibliography from a .bib file by 
%% running pdflatex/bibtex/pdflatex/pdflatex and then pasting the .bbl file
%% between \begin{thebibliography} and \end{bibliography}

%%% AUTHOR: Include a short description of each author following the
%%% structure below. Use the same short tags used previously.  
%%% Use \imageat{} and \imagedot{} instead of "@" and "." in
%%% email addresses-this replaces the symbols with graphics to avoid 
%%% e-mail address harvesting from the .pdf file
\begin{dajauthors}
\begin{authorinfo}[isaac]
    Isaac Rajagopal\\
  Massachusetts Institute of Technology\\
  Cambridge, Massachusetts\\
  isaacraj\imageat{}mit\imagedot{}edu \\
\end{authorinfo}
\end{dajauthors}

\end{document}